\documentclass[a4paper,12pt,oneside]{article}

\usepackage{graphicx}
\usepackage[all]{xy}
\usepackage[centertags]{amsmath}
\usepackage{latexsym}
\usepackage{xcolor}
\usepackage{amsfonts}
\usepackage{tikz-cd}
\usepackage{pst-node}
\usepackage{amssymb,amsthm}
\usepackage[english]{babel}

\usepackage{newlfont}

\usepackage{indentfirst}
\usepackage[english]{babel}
\usepackage[latin1]{inputenc}
\usepackage{newlfont}

\usepackage{indentfirst}
\usepackage[english]{babel}
\usepackage[latin1]{inputenc}
\usepackage{eufrak}
\usepackage{hyperref}

\newtheorem{remark}{Remark}[section]

\theoremstyle{plain}
\newtheorem{lemma}{Lemma}[section]

\newtheorem{proposition}{Proposition}[section]
\newtheorem{theorem}{Theorem}[section]
\newtheorem{definition}{Definition}[section]
\newtheorem{corollary}{Corollary}[section]

\newtheorem{example}{Example}[section]

\newcommand{\G}{\mathcal{G}}

\newcommand{\beqn}{\begin{eqnarray}}
\newcommand{\eeqn}{\end{eqnarray}}
\newcommand{\beq}{\begin{eqnarray}}
\newcommand{\eeq}{\end{eqnarray}}

\newcommand{\bpro}{\begin{proposition}}

\newcommand{\epro}{\end{proposition}}
\newcommand{\blem}{\begin{lemma}}
\newcommand{\elem}{\end{lemma}}
\newcommand{\bdfn}{\begin{definition}}
\newcommand{\edfn}{\end{definition}}
\newcommand{\bcor}{\begin{corollary}}
\newcommand{\ecor}{\end{corollary}}
\newcommand{\bthm}{\begin{theorem}}
\newcommand{\ethm}{\end{theorem}}
\newcommand{\bex}{\begin{example}}
\newcommand{\eex}{\end{example}}
\newcommand{\brmq}{\begin{remark}}
\newcommand{\ermq}{\end{remark}}
\newcommand{\benum}{\begin{enumerate}}
\newcommand{\eenum}{\end{enumerate}}
\newcommand{\bitem}{\begin{itemize}}
\newcommand{\eitem}{\end{itemize}}

\theoremstyle{plain}

\usepackage[latin1]{inputenc}

\title{Cartan-Schouten metrics for information geometry and machine learning}

  \author{ Andr\'e Diatta$^{( 1)}$\footnote{ \footnotesize \noindent (1) Aix-Marseille Univ, CNRS, Centrale Marseille, Institut Fresnel, 13013 Marseille, France.
\newline Email: andre.diatta@fresnel.fr; andrediatta@gmail.com.}
;  Bakary Manga$^{( 2)}$
  and  
Fatimata Sy$^{( 2)}$\footnote{\footnotesize \noindent(2) D\'epartement de Math\'ematiques et Informatique,
Universit\'e Cheikh Anta Diop de Dakar,
BP 5005 Dakar-Fann, Dakar, S\'en\'egal. Email: bakary.manga@ucad.edu.sn; syfatima89@gmail.com
} 
  }  

\begin{document}
\maketitle

\begin{abstract}
We study Cartan-Schouten metrics, explore invariant dual  connections, and propose them as models for Information Geometry. 
 Based on the underlying Riemannian barycenter and the biinvariant mean of Lie groups, we subsequently propose a new parametric  mean  for data science and machine learning which comes with several advantages compared to traditional tools  such as the arithmetic mean,     median, mode,  expectation,  least square method, maximum likelihood, linear regression.                         
We call a metric on a Lie group, a Cartan-Schouten metric, if its Levi-Civita connection is biinvariant, so  every 1-parameter subgroup through the unit is a geodesic. 
 Except for not being left or right invariant in general, Cartan-Schouten metrics enjoy the same geometry as biinvariant metrics, since they share the same Levi-Civita connection. 
To bypass the non-invariance apparent drawback, we show that Cartan-Schouten metrics are completely determined by their value at the unit. We give an explicit formula for recovering them from their value at the unit, thus making them much less computationally demanding, compared to general metrics on manifolds. Furthermore,  Lie groups with Cartan-Schouten metrics are complete Riemannian or pseudo-Riemannian manifolds.   We give a   complete characterization of Lie groups with Riemannian or Lorentzian Cartan-Schouten metrics.                                                     
Cartan-Schouten metrics  are in abundance on 2-nilpotent Lie groups. Namely,  on every 2-nilpotent Lie group, there is a 1-1 correspondence between the set of left invariant metrics and that of Cartan-Schouten metrics. 
\end{abstract}


\tableofcontents

\section{Introduction}

Information geometry is a field that sets the geometric framework for a deeper  understanding of information theory. It applies the concepts and techniques of differential geometry to statistics and probability theory.  Namely,
 families of probability distributions are studied as Riemannian manifolds (the Riemannian metric, say $\mu,$ being for instance the Fisher information metric) provided with some additional structure, a connection such that the covariant derivative of the Riemannian  metric is a totally symmetric $3$-tensor                          
(the Amari-Chentov tensor). Then a parameter family of pairs of connections $(\nabla^{\alpha},\nabla^{-\alpha})$ which are mutually dual with respect to the metric can be deduced such that the covariant derivatives  $\nabla^{\alpha}\mu,$ $\nabla^{-\alpha}\mu$      are
totally symmetric $3$-tensors. Conversely any totally symmetric $3$-tensor gives rise to such connections. See Section \ref{sect:fisher}.   

Generalizing the  above,  a statistical manifold is a Riemannian or pseudo-Riemannian manifold $(M,\mu)$, together with a (locally) flat torsion free linear connection
    $\nabla^1$ such that the covariant derivative $\nabla^1\mu=:S$ is a totally symmetric $(0,3)$-tensor. There is a unique connection $\nabla^{-1}$ dual to $\nabla^1$ with respect  to $\mu.$                         
                  The mutual duality  is equivalent to the following 
\beqn
Z\cdot\mu(X,Y) = \mu(\nabla_Z^1X,Y) + \mu(X,\nabla_Z^{-1} Y),
\eeqn for any vector fields $X,Y$ on $M.$ 
Letting $\nabla$  stand for the Levi-Civita connection of $\mu,$ one deduces a parameter family of pairs of mutually dual torsion free (locally) flat connections $(\nabla^{\alpha},\nabla^{-\alpha}),$ $\alpha\in\mathbb R,$ by
\beqn
\mu(\nabla_X^1Y,Z)&:=&\mu(\nabla_XY,Z)-\frac{1}{2}S(X,Y,Z)\nonumber\\
\mu(\nabla_X^{-1}Y,Z)&:=&\mu(\nabla_XY,Z)+\frac{1}{2}S(X,Y,Z)
\label{duality-3-tensor}
 \eeqn 
such that the covariant derivatives  $\nabla^{\alpha}\mu=-\nabla^{-\alpha}\mu$   are totally symmetric \cite{lauritzen87}, \cite{Matsuzoe2007}.   Such a generalization is not a mere mathematical theory, since it has indeed been proved that  every statistical manifold corresponds to a statistical model  \cite{van-le2005,matumoto1993}.             
The notion of dual connections, which was borrowed from  affine geometry (see e.g.\cite{lauritzen87}, \cite{Matsuzoe2010})  and introduced into information geometry by Amari \cite{amari85}, is now widely used as a  key tool in the applications of information geometry.
While metrics such as the Fisher information metric and divergences like the Kullback-Liebler divergence provide ways of measuring distances or differences between different probability distributions, geodesics appear as the shortest distances between probability distributions. All these tools strengthen the central role of information geometry in data science (see e.g.\cite{amaribook2016,amari85}).

In the present work, we study Cartan-Schouten metrics and propose them as models for information geometry.
A Cartan-Schouten metric on a Lie group $G$, is a metric whose Levi-Civita connection is the canonical biinvariant  Cartan-Schouten connection  \cite{Cartan-Schouten}, defined
by $\nabla_xy=\frac{1}{2} [x,y],$ for any $x,y$ in the Lie algebra $\mathcal G$ of G. So  every 1-parameter subgroup through the unit is a geodesic. 
Lie groups  with  Cartan Schouten metrics are geodesically (hence metrically) complete manifolds as well as realizing the 1-parameter subgroups as geodesics, which allows for a natural way of doing statistics and information geometry, using their intrinsic connection (just as is the case for biinvariant  metrics, see \cite{pennec1,miolane-pennec2015,pennec-biinvariant-means}).  In low dimension, Lie groups with a Cartan-Schouten metric have been classified, up to dimension 6, in \cite{bi-invariant-and-noninvariant-metrics-ghanam-hindeleh-thompson, rawashdeh-thompson, thompson3D}. 
The study of Cartan-Schouten metrics  is also motivated by the use of Riemannian statistics on Lie groups in the traditional fields such as  Machine Learning, Anatomy   
(\cite{miolane-pennec2015}, \cite{pennec1}).
Here are the main results and organization of the paper. In Section \ref{chap:Onstatistical-structures-on-groups},  we delve  into the  discussions on statistical manifolds and information geometry in details, directly deriving the Fisher information metric, dual connections, etc. We also set up  the framework, the definitions and discuss the motivations.                                                                      
Another interesting feature is that, although Cartan-Schouten metrics are in general not left nor right invariant, their geometry is biinvariant, since it is conveyed by their  Levi-Civita connection which is biinvariant. More interestingly, we provide a formula that uses the sole value of any Cartan-Schouten metric at the  unit (neutral element) to give its value everywhere, just like for left (or right) invariant metrics (see Theorem \ref{theo:exp-diffeo} and  Theorem \ref{theo:2-nilpotent}). This allows to bypass the noninvariance drawback and makes Cartan-Schouten metrics computationally  more attractive.
In Theorem \ref{theo:Riemann-Cartan-Schouten}, we describe solvable non-Abelian Lie groups with a Riemannian Cartan-Schouten metric. More
precisely we prove that a solvable non-Abelian Lie group admits a positive (or negative) definite Cartan-Schouten metric if
and only if it is 2-nilpotent. Section \ref{sec:Lorentzian-Carten-Schouten} is devoted to the discussion and complete characterization of Lie groups with a Cartan-Schouten metric $\mu$ of Lorentz type   (Lorentzian, for short), that is, $\mu$ is of signature $(1,n-1).$  Unlike biinvariant Lorentzian metrics which are      rather scarce (roughly speaking, they only exit on oscillator groups and on the special linear group SL(2) \cite{medina85}), Lie groups with Lorentzian Cartan-Schouten metrics are in abundance (see  Theorem \ref{thm:lorentz-Cartan-Schouten}, Theorem \ref{thm:metrics-oscilator} and   Theorem \ref{sect:2-nilpotent}).           
We expect Lorentzian Cartan-Schouten metrics studied in Section \ref{sec:Lorentzian-Carten-Schouten} to offer good models for singular statistical learning theory as in \cite{watanabe2009}, and  applications to relativity or quantum physics. 
 In any 2-nilpotent Lie group, Cartan-Schouten metrics can have any desired signature, and in fact, they are as many 
as metrics (of any signature) on the corresponding Lie algebra, see Theorem  \ref{theo:SetOfCartan-SchoutenMetrics} and Theorem \ref{theo:2-nilpotent}. And yet, all these Cartan-Schouten metrics share the same Levi-Civita connection and hence the same geometry. However, they may infer different  statistics, since for a fixed totally symmetric 3-tensor taken to be the Amari-Chentsov tensor, the corresponding parameter family of dual $\alpha$-connections $(\nabla^{\alpha},\nabla^{-\alpha})$, depends of the chosen Cartan-Schouten metric. When the connection coincides with the Levi-Civita connection of the metric, then the information geometry coincides with the Riemannian statistics   (or  statistics in Riemannian manifolds), which is a field in mathematics.              
In Theorem \ref{thm:general-compatible-metric-on-h-type-group} we describe all Cartan-Schouten metrics on H-type Carnot groups and supply their explicit expressions. 
Note that in 2-nilpotent Lie groups, Cartan-Schouten metrics have very simple expressions. Indeed, in global affine (exponential) coordinates, all their coefficients are polynomials of degree 2.                
 We explore and describe left invariant and biinvariant dual connections  and statistics   in Section \ref{sect:biinvariant-dual-connections-statistics}, see                                                              
Proposition \ref{prop:left-invariant-Cartan-Schouten-statistics}, Proposition  \ref{prop:biinvariantconnections}, Theorem \ref{theo:biinvariant-connections}.
In Section \ref{sect:newmodel}, we propose a new model of  parametric  mean for information geometry, data  science and machine learning. On the one hand, it is the common Riemannian barycenter of all the Cartan-Schouten metrics on Carnot groups of H-type.
On the other hand, it is also the  biinvariant exponential barycenter of  those Carnot groups of H-type. So it is at  the interplay between Riemannian statistics and Lie group invariant statistic. Furthermore, it combines the arithmetic mean and the variance,
and enjoys   a manifold  of parameters  to choose from.                                      

We expect  this work to  foster news  routes  of  research and  applications in several areas of science and technology,   within the scope  of applications of    Information     Geometry.                        
                              
Throughout the present work, unless  explicitly stated, the word metric refers to both Riemannian and pseudo-Riemannian  metrics.     We will let $\delta_{j,k}$  stand  for the  Kronecker symbol, with $\delta_{j,k}=1$ if $k=j$ and $0$ otherwise. Unless otherwise explicitly stated, $\nabla$ often stands for the canonical Cartan-Schouten connection, with $\nabla_xy:=\frac{1}{2}[x,y]$ for                  $x,y$ in   the Lie algebra at hand.

\section{On Information  Geometry  on Lie groups}\label{chap:Onstatistical-structures-on-groups}
\subsection{Fisher information metric,  Amari-Chentsov 3-tensor, $\alpha$-connections\label{sect:fisher}}
Let  $\Omega$ be a mesurable subset of  $ \mathbb R^m$ and $ \mathcal U$ a domain in $ \mathbb R^n.$  In $\Omega$, 
 consider a  familly of probabilities $p(-,\theta):\Omega \to \mathbb R,$ parametrized by
$\theta=(\theta_1,\dots\theta_n)\in\mathcal U$, such that the following hold :
(1) the family $p(-,\theta)$ is smooth with respect to $\theta,$
(2) $p(x,\theta)> 0,$ for any $x\in\Omega$ and any $\theta\in \mathcal U$  and (3) $\int_{\Omega}p(x,\theta)dx=1,$ for any $\theta\in \mathcal U.$ Here we have let   $\int_{\Omega}p(x,\theta)dx$ also stand for the  sum  $\sum\limits_{x\in\Omega}p(x,\theta)$ when  $\Omega$ is a discrete set.  
The Fisher information matrix associated to the familly $p(-,\theta)$ is the symmetric matrix $[\mu_{i,j}(\theta)]$ 
given by 
\beqn \label{Ficher-metric} \mu_{ij}(\theta)&:=&\int_{\Omega}\Big(\frac{\partial}{\partial \theta_i} \log p(x,\theta) \Big)\Big( \frac{\partial}{\partial \theta_j} \log p(x,\theta)\Big	)p(x,\theta) dx\nonumber\\
&=&E_{\theta}\Big(\frac{\partial}{\partial \theta_i} \log p(x,\theta) \;\frac{\partial}{\partial \theta_j} \log p(x,\theta)\Big	)\;,\eeqn
where  $E_{\theta}(f):= \int_{\Omega}fp(x,\theta) dx$ 
is the expectation of $f:\Omega\to  \mathbb R,$ with respect to $p(-,\theta).$

One notes that the  matrix $[\mu_{i,j}(\theta)]$ is positive semi-definite, since for any $Y=(y_1,\cdots,y_n)$ in $\mathbb  R^n$, one has: 
\beqn \label{semi-definite} Y[\mu_{i,j}(\theta)]Y^T&=&\sum\limits_{i,j=1}^n\mu_{ij}(\theta)y_iy_j\nonumber\\
&=&\int_{\Omega}\Big(\sum\limits_{i=1}^ny_i\frac{\partial}{\partial \theta_i} \log p(x,\theta) \Big)\Big( \sum\limits_{j=1}^ny_j\frac{\partial}{\partial \theta_j} \log p(x,\theta)\Big	)p(x,\theta) dx\nonumber\\
&=&E_{\theta}\Big(\sum\limits_{i}y_i\frac{\partial}{\partial \theta_i} \log p(x,\theta) \Big	)^2\ge 0\;.\eeqn

Hence $ \mu_{ij}(\theta)$ gives rise to a (possibly  pseudo-) Riemannian metric $\mu$ defined as $ \mu(x,y)= X[\mu_{i,j}(\theta)]Y^T,$ for any tangent vector field $x,y$ with (local) components $X,Y\in\mathbb R^n$. The metric (\ref{Ficher-metric}) is called the Fisher metric on $\mathcal M$, when it is definite positive \cite{amari85}, \cite{lauritzen87}, \cite{Matsuzoe2010}.

We further suppose that the derivation with respect to $\theta$ and the integration with respect to $x$ commute.  The latter property is always satisfied under a few conditions. Differentiating the equality $\int_{\Omega}p(x,\theta)dx=1,$  gives
\beqn 0&=&\frac{\partial}{\partial \theta_i}\int_{\Omega}p(x,\theta)dx=\int_{\Omega}\frac{\partial}{\partial \theta_i}p(x,\theta)dx =
\int_{\Omega}\Big(\frac{\partial}{\partial \theta_i}\log p(x,\theta)\Big)p(x,\theta)dx \nonumber\\
&=&E_{\theta}(\frac{\partial}{\partial \theta_i}\log p(x,\theta)).
\eeqn
The second derivatives lead to
\beqn 0&=&\frac{\partial^2}{\partial \theta_j\partial \theta_i}\int_{\Omega}p(x,\theta)dx=
\int_{\Omega}\Big(\frac{\partial^2}{\partial \theta_j\partial \theta_i}\log p(x,\theta)\Big)p(x,\theta)dx\nonumber\\
&+& \int_{\Omega}\Big(\frac{\partial}{\partial \theta_j}\log p(x,\theta) \; \frac{\partial}{\partial \theta_i}\log p(x,\theta)\Big)p(x,\theta)dx
\nonumber\\
&=&E_{\theta}\Big(\frac{\partial^2}{\partial \theta_j\partial \theta_i}\log p(x,\theta)\Big)+ \mu_{ij}(\theta).
\eeqn
Hence, one gets the following
\beqn
 \mu_{ij}(\theta)=-E_{\theta}\Big(\frac{\partial^2}{\partial \theta_j\partial \theta_i}\log p(x,\theta)\Big).
\eeqn
Differentiating (\ref{Ficher-metric}) with respect to $\theta_k,$ one gets

\beqn \label{Ficher-metric1} \frac{\partial}{\theta_k}\mu_{ij}(\theta)&:=&\int_{\Omega}\Big(\frac{\partial^2}{\partial \theta_k\partial \theta_i} \log p(x,\theta) \Big)\Big( \frac{\partial}{\partial \theta_j} \log p(x,\theta)\Big)p(x,\theta) dx\nonumber\\
&&+
\int_{\Omega}\Big(\frac{\partial}{\partial \theta_i} \log p(x,\theta) \Big)\Big( \frac{\partial^2}{\partial \theta_k\partial \theta_j} \log p(x,\theta)\Big)p(x,\theta) dx\nonumber\\
&&+
\int_{\Omega}\Big(\frac{\partial}{\partial \theta_i} \log p(x,\theta) \Big)\Big( \frac{\partial}{\partial \theta_j} \log p(x,\theta)\Big)\Big(
 \frac{\partial}{\partial \theta_k} \log p(x,\theta) \Big) p(x,\theta)dx .
\eeqn

Now we re-write (\ref{Ficher-metric1}) as 
\beqn \label{Ficher-metric2} \frac{\partial}{\theta_k}\mu_{ij}(\theta)&=&E_{\theta}\Big(\frac{\partial^2}{\partial\theta_k\partial \theta_i} \log p(x,\theta) \;\frac{\partial}{\partial \theta_j} \log p(x,\theta)\Big)
\nonumber\\
&&+
E_{\theta}\Big(\frac{\partial}{\partial \theta_i} \log p(x,\theta)  \frac{\partial^2}{\partial \theta_k\partial \theta_j} \log p(x,\theta)\Big) +S_{i,jk}(\theta)
,\eeqn
where $S$ stands for the Amari-Chentsov symmetric 3-tensor, with coefficients
\beqn 
S_{i,jk}(\theta):=E_{\theta}\Big(\frac{\partial}{\partial \theta_i} \log p(x,\theta)  \frac{\partial}{\partial \theta_j} \log p(x,\theta)
 \frac{\partial}{\partial \theta_k} \log p(x,\theta) \Big)\;.
\eeqn

Bringing Levi-Civita connection $\nabla$ of $\mu$ into play, we use the formula
\beqn \label{Ricci1} \frac{\partial}{\partial \theta_i}\mu_{jk}&=&\mu(\nabla_{\frac{\partial}{\partial \theta_i}}\frac{\partial}{\partial \theta_j},\frac{\partial}{\partial \theta_k})+\mu(\frac{\partial}{\partial \theta_j},\nabla_{\frac{\partial}{\partial \theta_i}}\frac{\partial}{\partial \theta_k})
,
\eeqn
to deduce  the following
\beqn \label{civita3}\frac{\partial}{\partial \theta_i}\mu_{jk}+\frac{\partial}{\partial \theta_j}\mu_{ik}-\frac{\partial}{\partial \theta_k}\mu_{ij}&=&2\mu(\nabla_{\frac{\partial}{\partial \theta_i}}\frac{\partial}{\partial \theta_j},\frac{\partial}{\partial \theta_k}).
\eeqn
Plugging (\ref{Ficher-metric2}) into the left hand side of (\ref{civita3}) we get 
\beqn \label{civita4}  E_{\theta}\Big(\frac{\partial^2}{\partial\theta_k\partial \theta_i} \log p(x,\theta) \;\frac{\partial}{\partial \theta_j} \log p(x,\theta)\Big)= \mu(\nabla_{\frac{\partial}{\partial \theta_i}}\frac{\partial}{\partial \theta_j},\frac{\partial}{\partial \theta_k})
-\frac{1}{2}S_{ijk}.
\eeqn
The right hand side of (\ref{civita4}) defines a connection $\nabla^1$ by  
\beqn \label{alpha-connections1}  \mu(\nabla_{\frac{\partial}{\partial \theta_i}}^{1}\frac{\partial}{\partial \theta_j},\frac{\partial}{\partial \theta_k})= \mu(\nabla_{\frac{\partial}{\partial \theta_i}}\frac{\partial}{\partial \theta_j},\frac{\partial}{\partial \theta_k})
-\frac{1}{2}S_{ijk},
\eeqn
$i,j,k=1,\dots,n.$ Since $S$ is symmetric, $\nabla^1$ is torsion-free. The torsion-free connection $\nabla^{-1}$ defined by  
\beqn \label{alpha-connections-m1}  \mu(\nabla_{\frac{\partial}{\partial \theta_i}}^{-1}\frac{\partial}{\partial \theta_j},\frac{\partial}{\partial \theta_k})= \mu(\nabla_{\frac{\partial}{\partial \theta_i}}\frac{\partial}{\partial \theta_j},\frac{\partial}{\partial \theta_k})
+\frac{1}{2}S_{ijk},
\eeqn
satisfies
\beqn \label{alpha-connections1m1}  \mu(\nabla_{\frac{\partial}{\partial \theta_i}}^{1}\frac{\partial}{\partial \theta_j},\frac{\partial}{\partial \theta_k})+
\mu(\nabla_{\frac{\partial}{\partial \theta_i}}^{-1}\frac{\partial}{\partial \theta_j},\frac{\partial}{\partial \theta_k})
= 2\mu(\nabla_{\frac{\partial}{\partial \theta_i}}\frac{\partial}{\partial \theta_j},\frac{\partial}{\partial \theta_k}) 
\eeqn
and
\beqn \label{alpha-connections1m1b}  \mu(\nabla_{\frac{\partial}{\partial \theta_i}}^{1}\frac{\partial}{\partial \theta_j},\frac{\partial}{\partial \theta_k})+
\mu(\frac{\partial}{\partial \theta_j},\nabla_{\frac{\partial}{\partial \theta_i}}^{-1}\frac{\partial}{\partial \theta_k})&=&
 \mu(\nabla_{\frac{\partial}{\partial \theta_i}}\frac{\partial}{\partial \theta_j},\frac{\partial}{\partial \theta_k})+
\mu(\frac{\partial}{\partial \theta_j},\nabla_{\frac{\partial}{\partial \theta_i}}\frac{\partial}{\partial \theta_k})\nonumber\\
&=&\frac{\partial}{\partial \theta_i}\mu_{jk}\;.
\eeqn
Equality (\ref{alpha-connections1m1}) is equivalent to 
\beqn \label{alpha-connections2a} \nabla= \frac{1}{2}(\nabla^{1}+\nabla^{-1}) 
\eeqn
whereas
Equality (\ref{alpha-connections1m1b}) is equivalent to 
\beqn \label{alpha-connections2b}X\cdot\mu(Y,Z)=  \mu(\nabla_{X}^{1}Y,Z) +  \mu(Y,\nabla_{X}^{-1}Z)\;,
\eeqn
for any vector fields $X,Y,Z$ on $\mathcal M.$
More generally, the torsion-free connections $\nabla^{\alpha},$ $\nabla^{-\alpha},$ $\alpha\in\mathbb  R,$  defined  for any $i,j,k=1,\dots,n,$ by 
\beqn \label{alpha-connections2}  \mu(\nabla_{\frac{\partial}{\partial \theta_i}}^{\alpha}\frac{\partial}{\partial \theta_j},\frac{\partial}{\partial \theta_k})&=& \mu(\nabla_{\frac{\partial}{\partial \theta_i}}\frac{\partial}{\partial \theta_j},\frac{\partial}{\partial \theta_k})
-\frac{\alpha}{2}S_{ijk}\;,\nonumber
\\
  \mu(\nabla_{\frac{\partial}{\partial \theta_i}}^{-\alpha}\frac{\partial}{\partial \theta_j},\frac{\partial}{\partial \theta_k})&=& \mu(\nabla_{\frac{\partial}{\partial \theta_i}}\frac{\partial}{\partial \theta_j},\frac{\partial}{\partial \theta_k})
+\frac{\alpha}{2}S_{ijk}
\eeqn
are re mutually dual with respect to $\mu,$ that is,
\beqn \label{alpha-connections3}X\cdot\mu(Y,Z)=  \mu(\nabla_{X}^{\alpha}Y,Z) +  \mu(Y,\nabla_{X}^{-\alpha}Z)\;,
\eeqn
 for any vector fields $X,Y,Z$ on $\mathcal M,$ and they satisfy $\nabla= \frac{1}{2}(\nabla^{\alpha}+\nabla^{-\alpha}), $  for any $\alpha\in\mathbb R.$
  Furthermore, the $3$-tensors $\nabla^{\alpha} \mu$  and $\nabla^{-\alpha} \mu$ are totally symmetric. More precisely, we have the following identities 
\beqn \nabla^{\alpha}\mu(X,Y,Z)&=&X\cdot\mu(Y,Z)-\mu(\nabla_{X}^{\alpha}Y,Z) -  \mu(Y,\nabla_{X}^{\alpha}Z)\nonumber\\
&=&X\cdot\mu(Y,Z)-\Big(\mu(\nabla_{X}Y,Z)- \frac{\alpha}{2}S(X,Y,Z)\Big)\nonumber\\
&&-\Big(\mu(Y,\nabla_{X}Z)- \frac{\alpha}{2}S(X,Y,Z)\Big) =\alpha S(X,Y,Z)
\eeqn
and 
\beqn \nabla^{-\alpha}\mu(X,Y,Z)&=&X\cdot\mu(Y,Z)-\mu(\nabla_{X}^{-\alpha}Y,Z) -  \mu(Y,\nabla_{X}^{-\alpha}Z)\nonumber\\
&=&X\cdot\mu(Y,Z)-\Big(\mu(\nabla_{X}Y,Z)+ \frac{\alpha}{2}S(X,Y,Z)\Big)\nonumber\\
&&-\Big(\mu(Y,\nabla_{X}Z)+ \frac{\alpha}{2}S(X,Y,Z)\Big) =-\alpha S(X,Y,Z)\;.
\eeqn

Along these lines, one defines a statistical manifold as a triplet $(M,\mu,\nabla^1)$  where $M$ is a differential manifold, $\mu$ and $\nabla^1$
 are respectively a Riemannian or pseudo-Riemannian metric and a torsion free linear connection on $M$ such that the covariant derivative
 $\nabla^1 \mu$  of $\mu$ with respect to $\nabla^1,$ is totally
symmetric.  The pair $(\mu,\nabla^1)$ is then also termed a 
Codazzi structure or a statistics on $M,$ while $\mu$ and $\nabla^1$ are said to be compatible. The totally symmetric  $3$-tensor $S:=\nabla^1\mu$  is termed the Amari-Chentsov tensor. There is a unique torsion free  affine connection $\nabla^{-1}$ on $M,$ such that 
\beqn\label{dual-connections}
X\cdot\mu(Y,Z)=\mu(\nabla_X^{1}Y,Z)+\mu(Y,\nabla_X^{-1}Z)\;,\eeqn
 for any vector fields  $X,Y,Z\in\mathfrak X(M)$. Two connections  $\nabla^1$ and $\nabla^{-1}$ satisfying (\ref{dual-connections}), are said to be mutually dual (or just dual, for short) with respect to $\mu.$ 
Conversely, in a (pseudo-) Riemannian manifold ($M,\mu$) with Levi-Civita (LC, for short) connection $\nabla,$ consider a totally symmetric covariant $3$-tensor $S$. The tensor field $A$  defined by 
\beqn
\mu(A(X,Y),Z):=S(X,Y,Z), 
\eeqn
for any $X,Y,Z\in\mathfrak X(M),$
 gives rise to the two torsion free linear connections $\nabla^1:=\nabla-\frac{1}{2} A$  and $\nabla^{-1}:=\nabla+\frac{1}{2} A$, which satisfy (\ref{dual-connections}) and $\nabla^1\mu=-\nabla^{-1}\mu=S$. 
 This thus establishes an equivalence between statistical structures and totally symmetric $3$-tensors on M.  One defines a parameter family of  manifolds, the $\alpha$-manifolds $(M,\mu,\nabla^\alpha,\nabla^{-\alpha}),$ where the $\alpha$-connections are given by
\beqn
\mu(\nabla_X^\alpha Y,Z)&=&\mu(\nabla_X Y,Z)-\frac{\alpha}{2}S(X,Y,Z)\;,\nonumber\\
 \mu(\nabla_X^{-\alpha} Y,Z)&=&\mu(\nabla_X Y,Z)+\frac{\alpha}{2}S(X,Y,Z),  
\eeqn
or equivalently 
\beqn  \nabla^\alpha=\nabla-\frac{\alpha}{2} A\;,\;\nabla^{-\alpha}=\nabla+\frac{\alpha}{2} A\;,
\eeqn
for any $X,Y,Z\in \mathfrak X(M).$ The connections $\nabla^\alpha$ and $\nabla^{-\alpha}$ are dual with respect to $\nabla.$
So another equivalent definition of a statistical manifold, is a Riemannian (or more generally pseudo-Riemannian)  manifold, together with a totally symmetric tensor $S.$

 In this generalization, totally geodesic submanifolds now play the role of affine
subspaces of Eulidean space in standard statistical models, while geodesics replace straight lines.
It has been proved (\cite{van-le2005,matumoto1993}) that  every statistical manifold is indeed a statistical model.       
\subsection{Information  geometry using  Cartan-Schouten metrics}
Let us remind that Cartan-Schouten connections on a Lie group $G,$ are the left invariant connections such that every $1$-parameter subgroup of $G$ through the identity $\epsilon$ is a geodesic. 
The classical  $+$, $-$ and $0$ Cartan-Schouten connections (\cite{Cartan-Schouten}) are respectively given in the Lie algebra $\mathcal G$ of $G$ by
\beqn\label{Cartan-connections} \nabla_x y=\lambda [x,y], \;\; \lambda=1,0,\frac{1}{2}, \;\;\forall \; x,y\in\mathcal G\;.
\eeqn
\begin{definition} 
We will refer to the $0$-connection given by  $\nabla_xy=\frac{1}{2} [x,y],$ $\forall \; x,y\in\mathcal G,$  as the Cartan-Schouten canonical connection. It is the unique symmetric (torsion free) Cartan-Schouten which is  bi-invariant.
\end{definition}

Our models of statistical manifolds of interest are couples ($M, \mu$),  
where $M:=G$ is a Lie group and $\mu$ is a Riemannian or pseudo-Riemannian metric on $G$,
 which is covariantly constant (or equivalently, parallel) with respect to the Cartan-Schouten 
canonical connection $\nabla$. The latter property simply reads $\nabla \mu=0$  or, equivalently,
\beqn \label{eq:parallel} x^+\cdot\mu(y^+,z^+)=\frac{1}{2}\Big(\mu([x^+,y^+],z^+)+\mu(y^+,[x^+,z^+])\Big) 
\eeqn
for any left invariant vector fields $x^+,y^+,z^+$ on $G.$ Importantly, when we take the dual connections to coincide with $\nabla,$ then the model coincides with both the (pseudo-) Riemannian Statistics and Lie group (bi)invariant statistics  as in \cite{miolane-pennec2015}, \cite{pennec1}, \cite{pennec2}.  

\begin{definition} \label{def:Cartan-Schouten-Metrics} If a metric  $\mu$ on a Lie Group, is parallel with respect to the Cartan-Schouten canonical connection, that is $\nabla\mu=0$, then we call it a Cartan-Schouten metric. Equivalently, a Cartan-Schouten metric is a metric whose Levi-Civita connection is the (biinvariant) Cartan-Schouten 
canonical connection $\nabla.$   
\end{definition}

Note that, although the two approaches of defining a statistical manifold given above are equivalent, 
in our case the choice of one may influence the direction of the study in the following way.
Indeed, since Cartan-Schouten metrics are not left nor right invariant in general, fixing a left invariant Amari-Chentsov
 3-tensor induces in general a non-invariant pair of dual connections which thus come
 with a non-invariant underlying statistical model.
In contrast, although it infers a non-invariant 3-tensor, the choice of a left invariant or biinvariant
 connection compatible with $\mu,$ leads to an invariant statistical model.

\subsection{Some advantages and more motivations}
As mentioned in the introduction, every  $1$-parameter subgroup through the unit is a geodesic. 
 Cartan Schouten metrics are geodesically (hence metrically) complete manifolds. So they are a  nice model for Riemannian statistics.
 Their geometries are invariant, since they are  conveyed by the biinvariant symmetric connection. 
 The geodesics of $\nabla$ are translates of one parameter subgroups,  left or right-invariant vector fields are Killing vector fields or affine collineation of
the connection $\nabla$. It has also been proposed, see e.g. \cite{miolane-pennec2015}, \cite{pennec1}, that Riemannian or pseudo-Riemannian metrics which are  biinvariant be used for Riemannian statistics on Lie groups. Cartan-Schouten metrics are a generalization of biinvariant metrics where the invariance assumption on the metric has been dropped. Since biinvariant metrics and Cartan-Shouten metrics share the same Levi-Civita connection (which is biinvariant), they thus share the same geometry. Their common (pseudo-) Riemannian mean coincides  with the Lie group biinvariant exponential barycenter. Thus
 Cartan-Schouten metrics allow for a much larger familly of Lie groups with a biinvariant geometry, as we shall see throughout this work.
 For example, Heisenberg groups $\mathbb H_{2n+1}$, $n\ge 1,$  carry infinitely many non-equivalent Cartan-Schouten metrics, although they do not carry any biinvariant metric (see Example \ref{Heisenberg}).
 Interestingly, those metrics can have any signature, and in fact, there are as many 
Cartan-Schouten metrics in any 2-nilpotent Lie group than metrics (of any signature) on its Lie algebra.
Another interesting geometric feature is that, there are infinitely many non isometric Cartan Schouten metrics that all share the same geometry (same Riemannian curvature, same holonomy group, same geodesics, same Ricci curvature, same sectional  curvature) and yet have different signatures. As a matter of fact, such a set of metrics covers all possible signatures $(p,q),$ with no restriction on the the integers $p,q\ge 0$ with $p+q=n.$ 
Cartan-Schouten metrics of Lorentz type studied in Section \ref{sec:Lorentzian-Carten-Schouten}  may also be of good interest for singular statistical learning theory (\cite{watanabe2009}) and applications to relativity and quantum physics.     

\subsection{Hypersurfaces, totally geodesic submanifolds}
In standard statistical models, affine spaces play important roles, for example in dimension reduction techniques (principal component analysis and factor analysis), linear models, optimization problems, etc.   
 In the generalization to Riemannian and pseudo-Riemannian statistics, totally geodesic submanifolds now play the role of affine
subspaces of Eulidean space in standard statistical models, while geodesics replace straight lines.
In that context, geodesics certainly play a crucial role, since they represent the most natural path between probability distributions. A geodesic submanifold $\tilde M$  of a manifold $(M,\nabla)$ endowed with a connection, is a submanifold such that $\nabla$ preserves  tangent vectors, so $\nabla$ restricts to a  connection in  $\tilde M.$ Thus, any geodesic of  $(M,\nabla)$  passing through a point of  $\tilde M$ stays entirely in $\tilde M,$  any geodesic in  $\tilde M$ is also a geodesic of $M.$ In other words, the geodesic equation of $M$ restricts to the geodesic  equation of $\tilde M.$    Geodesic submanifolds can  represent simpler constrained models,  the  subset of uncorrelated variables in the multivariate gaussian distributions. Totally geodesic submanifold  also provide a natural way of dimensionality reduction of parameter space. 

In  our models, totally geodesic submanifolds are closed subgroups. Non-degenerate subgroup inherit the Cartan-Schouten structure by restriction of the metric. This allows for a natural dimension reduction.

\section{Cartan-Schouten metrics on Lie groups}

\subsection{General results}
The following result from \cite{diatta-manga-sy-NLAGA} covers most of nonsolvable Lie groups, including  all semi-simple Lie groups and most of non-decomposable nonsolvable Lie groups.                                                                                                        
\begin{theorem} \label{Theorem:parallel-metric-perfect-Lie-groups}\cite{diatta-manga-sy-NLAGA} Let $G$ be a perfect Lie group, that is, its Lie algebra $\mathcal G$  satisfies $[\mathcal G,\mathcal G]=\mathcal G.$  
If $G$ has a Cartan-Schouten metric $\mu$, then $\mu$ is necessarily biinvariant.
 \end{theorem}
In particular, every Cartan-Schouten metric on a semi-simple Lie group $G$ or on its cotangent bundle $T^*G,$ is biinvariant,  where $T^*G$ is endowed with its Lie group structure induced by the right trivialization \cite{diatta-manga-sy-NLAGA}. 

We announce the following result whose complete proof will be published elsewhere. 
\begin{theorem}\label{theo:Cartan-Schouten-formula-Lie-Algebra}A connected Lie group $G$  has a Cartan-Schouten metric $\mu$ if and only if its Lie algebra, say $\mathcal G$,
 has  a metric $\bar \mu$ with same signature as $\mu,$ satisfying the following
\beqn \label{eq:Cartan-Schouten}0&=&\bar \mu([[x_1,x_2],y],z) +\mu(y,[[x_1,x_2],z])
\eeqn
for any $x_1,x_2,y,z\in\mathcal G.$ Furthermore, at the unit $\epsilon\in G$, the two metrics $\mu$ and $\bar \mu$ coincide.
\end{theorem}
For the proof that the Lie algebra of a Lie Group with a Cartan-Schouten metric, has   a metric satisfying (\ref{eq:Cartan-Schouten}), we will need Proposition \ref{prop:curvature-Levi-Civita}, which is a corollary of 
the following well-known Proposition \ref{prop:curvatures-of-dual-connections}, whose proof is provided below using a new method. If a metric on a Lie algebra satisfies (\ref{eq:Cartan-Schouten}),  we call it a Cartan-Schouten metric on the Lie algebra.  Equality (\ref{eq:Cartan-Schouten}) is equivalent to the fact that the adjoint operator ad$_u$ of any $u\in[\G,\G],$ is skew symmetric with respect to $\mu.$                             

\begin{proposition}\label{prop:curvatures-of-dual-connections}
Let $\mu$ be a Riemannian or pseudo-Riemannian metric on a manifold $M.$ If $\nabla^1$ and $\nabla^{-1}$ are two  connections on $M$ which are dual with respect to $\mu,$ then their respective curvature tensors $R^{\nabla^1}$ and $R^{\nabla^{-1}}$   satisfy the following equation,
\beqn \label{eq:curvatures-of-dual-connections} 0&=&\mu\Big(R^{\nabla^1} (X_1,X_2)Y,Z\Big)+\mu\Big(Y,R^{\nabla^{-1}} (X_1,X_2)Z \Big) \;,
\eeqn
for any $X_1,X_2,Y,Z\in\mathfrak{X}(M).$
\end{proposition}
\begin{proof} 
For any $Y,Z\in\mathfrak{X}(M),$ consider the function  $F_{Y,Z}=\mu(Y,Z)\in C^\infty(M,\mathbb R)$ and denote its differential by   $\nu:=dF_{Y,Z}.$ The relation
\beqn \label{eq:curvatures-dual-connections} 0=\nabla\mu (X,Y,Z)= X\cdot \mu(Y,Z)-\mu(\nabla_X^1Y,Z)-\mu(Y,\nabla_X^{-1}Z)),
\eeqn
is thus equivalent to 
$\nu (X)=\mu(\nabla_XY^1,Z)+\mu(Y,\nabla_X^{-1}Z))$ for any $X\in\chi(M).$
 Since $\nu$ is an exact 1-form, we thus get
\beqn \label{eq1a} 0=d\nu(X_1,X_2)&=&X_1\cdot \nu(X_2)-X_2\cdot \nu(X_1)-\nu([X_1,X_2])\nonumber\\
&=& X_1\cdot\mu(\nabla_{X_2}^1Y,Z)+ X_1\cdot\mu(Y,\nabla_{X_2}^{-1}Z)
- X_2\cdot\mu(\nabla_{X_1}^1 Y,Z)
\nonumber\\
&&- X_2\cdot\mu(Y,\nabla_{X_1}^{-1}Z)
-\mu(\nabla_{[X_1,X_2]}^1Y,Z)- \mu(Y, \nabla_{[X_1,X_2]}^{-1}Z),\eeqn
which we now expand, by applying (\ref{eq:curvatures-dual-connections})  to (\ref{eq1a}),  as
\beqn \label{eq1b} 0&=& \mu(\nabla_{X_1}^1\nabla_{X_2}^1Y,Z)+\mu(\nabla_{X_2}^1Y,\nabla_{X_1}^{-1}Z) 
+ \mu(\nabla_{X_1}^1Y,\nabla_{X_2}^{-1}Z) + \mu(Y,\nabla_{X_1}^{-1}\nabla_{X_2}^{-1}Z) \nonumber\\
&&- \mu(\nabla_{X_2}^1\nabla_{X_1}^1 Y,Z)- \mu(\nabla_{X_1}^1 Y,\nabla_{X_2}^{-1}Z)
- \mu(\nabla_{X_2}^1Y,\nabla_{X_1}^{-1}Z)- \mu(Y,\nabla_{X_2}^{-1}\nabla_{X_1}^{-1}Z)\nonumber\\
&&
-\mu(\nabla_{[X_1,X_2]}^1Y,Z)- \mu(Y, \nabla_{[X_1,X_2]}^{-1}Z).\eeqn

Equality (\ref{eq1b}) readily simplifies to 
\beqn\label{eq1c} 0&=&\mu\Big((\nabla_{X_1}^1\nabla_{X_2}^1- \nabla_{X_2}^1\nabla_{X_1}^1 -\nabla_{[X_1,X_2]}^1)Y,Z\Big)
\nonumber\\
&&
+\mu\Big(Y,(\nabla_{X_1}^{-1}\nabla_{X_2}^{-1} - \nabla_{X_2}^{-1}\nabla_{X_1}^{-1} - \nabla_{[X_1,X_2]}^{-1})Z \Big)
\nonumber\\
&=&\mu\Big(R^{\nabla^1} (X_1,X_2)Y,Z\Big)+\mu\Big(Y,R^{\nabla^{-1}} (X_1,X_2)Z \Big) \;.
\eeqn
\end{proof}

As a direct corollary of Proposition \ref{prop:curvatures-of-dual-connections}, we get the following result, which is also consequence of the Ambrose-Singer holonomy theorem.  
\begin{proposition}\label{prop:curvature-Levi-Civita}
Let $\mu$ be a Riemannian or pseudo-Riemannian metric and $\bar\nabla$ a (not necessarily torsion free) connection, on a manifold $M.$  Suppose $\bar\nabla\mu=0.$ Then the curvature tensor $R^{\bar\nabla}$ of $\bar\nabla$, is skew symmetric with respect to $\mu:$
\beqn \label{eq:skew-curvature-parallel} 0&=&\mu\Big(R^{\bar\nabla} (X_1,X_2)Y,Z\Big)+\mu\Big(Y,R^{\bar\nabla} (X_1,X_2)Z \Big) \;,
\eeqn
for any $X_1,X_2,Y,Z\in\mathfrak{X}(M).$
\end{proposition}
\begin{proof} In Proposition \ref{prop:curvatures-of-dual-connections}, take $\nabla^1=\nabla^{-1}=\bar\nabla,$ the identity (\ref{eq:curvatures-of-dual-connections} ) then becomes (\ref{eq:skew-curvature-parallel}).
\end{proof}

The proof that the Lie algebra of a Lie Group with a Cartan-Schouten metric, has   a metric satisfying (\ref{eq:Cartan-Schouten}) is deduced from  Proposition \ref{prop:curvature-Levi-Civita}. Indeed, if $M=G$ is a Lie group, $\G$ its Lie algebra, we take $X_1,X_2,Y,Z$ to be all left invariant vector fields on $G$ and $\bar\nabla$ to be the Cartan-Schouten standard connection. Then Equation (\ref{eq:skew-curvature-parallel} ) taken at the unit of $G$, gives Equation (\ref{eq:Cartan-Schouten}). The proof  that from a metric on $\G$ satisfying 
 Equation (\ref{eq:Cartan-Schouten}), we can construct a Cartan-Schouten metric on any connected Lie group  with Lie algebra $\G,$ will be published elsewhere.

\subsection{Riemannian Cartan-Schouten metrics}

In this section, we  show that the Lie groups with a Riemannian metric which is also a Cartan-Schouten metric, are essentially all the 2-nilpotent Lie groups,  the compact simple Lie groups, the Abelian Lie groups and all their Cartesian (direct) products.

The following result is due to J. Milnor.
\begin{theorem}\label{theo:milnor}\cite{Milnor76} The only Lie groups with a biinvariant Riemannian metric are the Cartesian product of some compact Lie group and  some additive vector group.
\end{theorem}
Milnor's result somehow  implicitly spelled disappointment for those who expected a lot of Lie groups with biinvariant Riemannian metrics for applications. However, we show here that if we drop the biinvariance property for the Riemannian metric, and only require that its Levi-Civita connection be biinvariant (which characterizes Cartan-Schouten metrics), then it turns out that a much wider family of Lie groups enjoy such a property. We actually further show that
Cartan Schouten metrics of any given signature are in abundance on 2-nilpotent Lie groups, as we shall see Section \ref{sect:2-nilpotent}.
\begin{theorem} \label{theo:Riemann-Cartan-Schouten}For a solvable non-Abelian Lie group G, the following are equivalent.
(1) G admits a positive (or negative) definite Cartan-Schouten metric. 
(2) G is 2-nilpotent. \end{theorem}
\begin{proof} Let $\mathcal G$ be the Lie algebra of G. Since $\mathcal G$ is solvable, the  linear transformation $ad_u$ is nilpotent for  any $u\in[\mathcal G,\mathcal G].$   Thus, if G
admits a Cartan-Schouten metric  which is positive (or negative) definite, $ad_u$ being   skew-adjoint (see (\ref{eq:Cartan-Schouten})), is necessarily equal to the zero map, for  any $u\in[\mathcal G,\mathcal G].$  Thus $\G$ is 2-nilpotent.
Conversely, any 2-nilpotent Lie group possess as many Cartan-Schouten metrics which are positive definite, as left (resp. right) invariant  Riemannian metrics. See Theorem  \ref{theo:2-nilpotent}.
\end{proof}

From Theorem \ref{Theorem:parallel-metric-perfect-Lie-groups} and Theorem \ref{theo:milnor}, we deduce following 

\begin{theorem} A perfect Lie group admits a Cartan-Schouten metric  which is positive (or negative) definite, if and only  if it is semisimple and compact. 
\end{theorem}
\begin{proof}
According to Theorem \ref{Theorem:parallel-metric-perfect-Lie-groups}, if a perfect  Lie group $G$ admits a Cartan-Schouten metric, then the latter  must be biinvariant.  Now from Theorem \ref{theo:milnor},  we deduce that  $G$ must  then be a compact  and semisimlple Lie group. 
\end{proof}

\subsection{Lorentzian Cartan-Schouten metrics}\label{sec:Lorentzian-Carten-Schouten}

\subsubsection{A general result on Lorentzian Cartan-Schouten metrics}
Recall that Lie groups with  a  biinvariant  metric  are those special    Lie groups with a Cartan-Schouten  metric $\mu$ which is invariant under both left and right translations of all the group elements. Equivalently, the Lie algebras of such Lie groups have a metric (the value of $\mu$  at the unit $\epsilon$) which is invariant under the adjoint operator (adjoint-invariant, or ad-invariant for short) of every element of the Lie algebra. 
\begin{theorem}\cite{medina85}
The only simply connected and nonsimple Lie groups which admit an indecomposable biinvariant Lorentz metric are the oscillator Lie groups $G_\lambda$.
\end{theorem}
Unlike biinvariant Lorentz metrics which only exist on the oscillator Lie groups $G_\lambda,$ the special linear group SL($2$), the Abelian Lie groups and their direct products, the more general Cartan-Schouten metrics which  are Lorentz metrics, exist in abundance.  Indeed, every 2-nilpotent  Lie group of dimension  $n,$ has as many Lorentzian  Cartan-Schouten metrics as there are scalar products of  signature ($1,n-1$) in   any vector space of dimension $n$ (see  Theorem \ref{theo:2-nilpotent}).      

We announce the following general result whose proof will be published elsewhere.
\begin{theorem}\label{thm:lorentz-Cartan-Schouten} Let G be a nondecomposable solvable Lie group. Suppose G has a Cartan Schouten  metric, which is a Lorentzian metric.Then G is one of the following types: (1) either 2-nilpotent or 3-nilpotent, (2) 3-step solvable and the derived ideal of its Lie algebra is a direct sum of a Heisenberg Lie algebra and an Abelian Lie algebra. \end{theorem}

\subsubsection{The oscillator Lie algebras and Lie groups}
The oscillator Lie algebras $\G_\lambda$ (resp. Lie groups $G_\lambda$) are an important family of Lie groups in mathematics and physics with a rich geometry \cite{medina85}.
Let $\lambda:=(\lambda_1,\dots,\lambda_n)$ be in $\mathbb R^n$ with $0<\lambda_1\leq \lambda_2\leq \dots\leq\lambda_n$ and let $G_\lambda$ stand for 
$\mathbb R^{2n+2}=\mathbb R\times\mathbb R\times \mathbb C^n$ endowed with the Lie group structure given by the following product. Let $\sigma=(t,s,z_1,\dots z_n), \tau=(t',s',z_1',\dots z_n')$ $\in  G_{\lambda}$,

\begin{eqnarray}\sigma\tau&=&\Big(t+t',s+s'+\frac{1}{2}\displaystyle\sum_{j=1}^nIm(\bar z_j)\exp(it\lambda_j)z_j'\; ,
z_1+ \exp(it\lambda_1)z_1'\; ,
\nonumber\\
&& \dots,z_j+ \exp(it\lambda_j)z_j',\dots,z_n+\exp(it\lambda_n)z_n'\Big)\;.\end{eqnarray}

Set $z_j=x_j+i y_j$, where $x_j, y_j\in \mathbb R$ and $i^2=-1,$ so that $\sigma$  and $\tau$ are respectively identified with 
 the $(2n+2)$-uplets $(t,s,x_1, \dots, x_n,y_1, \dots,y_n),$ $(t',s',x'_1, \dots x'_n,y'_1, \dots,y'_n,).$ Hence, the above product now reads
 \beqn
\sigma\tau&=&\Big(t+t', s+s'+\frac{1}{2}\sum_{j=1}^n(x_jy'_j-y_jx'_j)\cos(\lambda_jt)
 +\frac{1}{2}\sum_{j=1}^n(x_jx'_j+y_jy'_j)\sin(\lambda_jt);  \cr 
         & & x_1+x'_1\cos(\lambda_1t)-y'_1\sin(\lambda_1t), \dots, 
  x_n+x'_n\cos(\lambda_nt)-y'_n\sin(\lambda_nt); \cr 
   & & \cr 
  & &  y_1+y'_1\cos(\lambda_1t)+x'_1\sin(\lambda_1t) , \dots, y_n+y'_n\cos(\lambda_nt)+x'_n\sin(\lambda_nt)\Big).
\eeqn 
 
The Lie algebra $\mathcal G_\lambda$ of the oscillator group $G_\lambda$ has a basis $\Big(e_{-1},e_0,e_1,\dots,e_{2n}\Big)$ in which its Lie bracket reads
\begin{eqnarray}[e_j,e_{n+j}] =e_0\;,\; [e_{-1},e_j] =\lambda_je_{n+j}\;, \; [e_{-1},e_{n+j}] =-\lambda_je_{j}\;,\;j=1,\dots,n.\end{eqnarray}
Following \cite{medina85}, up to a scalar factor, there is a unique adjoint invariant metric $\langle,\rangle_\lambda$ on $\mathcal G_\lambda$ given, for $x=\sum\limits_{j=-1}^{2n}x_je_j,$  $y=\sum\limits_{j=-1}^{2n}y_je_j,$ by
\begin{eqnarray}\langle x,y\rangle_\lambda&=& x_{-1} y_{0}+y_{-1}x_{0}+ \displaystyle\sum_{j=1}^n \frac{1}{\lambda_j} \Big(x_{j} y_{j}+ x_{n+j}y_{n+j}\Big).
\end{eqnarray}
Note that the ad-invariant metric $\langle, \rangle_\lambda$ is a Lorentz metric.
Following Theorem \ref{theo:Cartan-Schouten-formula-Lie-Algebra}, in order to find all Cartan-Schouten metrics, we just need to directly look at the Lie algebra level.
We apply Equation (\ref{eq:Cartan-Schouten}) to get all the Cartan-Schouten metrics on $\mathcal G_\lambda.$ Note that the derived ideal  $[\mathcal G_\lambda,\mathcal G_\lambda]$ of $\mathcal G_\lambda,$ is spanned by  $(e_0,e_1,\dots,e_{2n}).$
When Equation (\ref{eq:Cartan-Schouten}) is satisfied in a trivial way, so that no information on $\mu$ can be read off, we simply skip it.   

\noindent
$\bullet $  $0=\mu([e_j,e_{-1}],e_{-1})+\mu(e_{-1},[e_j,e_{-1}])=-2\lambda_j \mu(e_{n+j},e_{-1}),$

\noindent
$\bullet $  $0=\mu([e_j,e_{-1}],e_0)+\mu(e_{-1},[e_j,e_{0}])=-\lambda_j \mu(e_{n+j},e_0),$

\noindent
$\bullet $    $0=\mu([e_j,e_{-1}],e_k)+\mu(e_{-1},[e_j,e_k])= -\lambda_j \mu(e_{n+j},e_k),$  

\noindent
$\bullet $    $0=\mu([e_j,e_{-1}],e_{n+k})+\mu(e_{-1},[e_j,e_{n+k}])=-\lambda_j\mu(e_{n+j},e_{n+k})+\delta_{j,k}\mu(e_{-1},e_{0}), $

\noindent
$\bullet $   $ 0=\mu([e_j,e_k],e_{-1})+\mu(e_{k},[e_j,e_{-1}])=  -\lambda_j\mu(e_{k},e_{n+j}),$

\noindent
$\bullet$      $ 0=\mu([e_j,e_{k}],e_{n+p})+\mu(e_{k},[e_j,e_{n+p}])= \delta_{j,p}\mu(e_{k},e_{0})=0,$

\noindent
$\bullet $    $0=\mu([e_j,e_{n+k}],e_{0})+\mu(e_{n+k},[e_j,e_{0}])=\delta_{j,k}\mu(e_{0},e_{0})=0,$

\noindent
$\bullet $      $0=\mu([e_j,e_{n+k}],e_{n+p})+\mu(e_{n+k},[e_j,e_{n+p}])=  \delta_{j,k}\mu(e_{0},e_{n+p})+ \delta_{j,p}\mu(e_{n+k},e_{0}),$

\noindent
$\bullet $    $0=\mu([e_{n+j},e_{-1}],e_{-1})+\mu(e_{-1},[e_{n+j},e_{-1}])= 2\lambda_j\mu(e_{j},e_{-1}),$  

\noindent
$\bullet $    $0=\mu([e_{n+j},e_{-1}],e_{k})+\mu(e_{-1},[e_{n+j},e_{k}])= \lambda_j\mu(e_{j},e_{k})-\delta_{j,k}\mu(e_{-1},e_0),$

We summarize the above in the
\begin{theorem}\label{thm:metrics-oscilator}
A metric on $\G_\lambda$ is a Cartan-Schouten metric if and only if its nonzero coefficients are as follows
\beqn\label{Cartan-Schouten-oscillator} \mu(e_{j},e_{j})= \mu(e_{n+j},e_{n+j})=\frac{1}{\lambda_j}\mu(e_{-1},e_{0})\neq 0\;,\; \mu(e_{-1},e_{-1})\;,\; j=1,\dots,n.  \eeqn 
Every metric as in (\ref{Cartan-Schouten-oscillator})  is a Lorentzian. It is ad-invariant if and only if  $\mu(e_{-1},e_{-1})=0.$
\end{theorem}

\subsection{Case where the exponential map is a diffeomorphism}

For many Lie groups, the exponential map  $\exp: \G\to G,$ is a diffeomorphism. This  is the case for every connected and  simply connected nilpotent Lie group and more generally, every  connected and  simply connected completely solvable Lie  group. More precisely, following the well known works of Dixmier and Saito in 1957, the exponential map is a diffeomorphism, if and only if $G$ is  a connected and  simply connected solvable Lie group  that does not  contain a closed subgroup isomorphic to the circle, the universal cover of the special linear group $SL(2,\mathbb R)$  of $2$ by $2$ real matrices with determinant $1,$
 the universal cover of the group $E(2):=O(2)\ltimes \mathbb R^2$ of rigid motions of the Euclidean $2$-space, or the 4-dimensional connected and simply connected oscillator Lie group with Lie algebra $\bar{\mathfrak{e}}$ having a basis $(x,y,H,z),$ with Lie bracket $[H,x]=y,$ $[H,y]=-x,$ $[x,y]=z.$  
Equivalently, the exponential map is a diffeomorphism, if and only if $G$ is  a connected and  simply connected solvable Lie group whose Lie algebra $\G$ does not  contain a Lie subalgebra isomorphic  the Lie algebra $\mathfrak{e}$ of $E(2):=O(2)\ltimes \mathbb R^2$ or $\bar{\mathfrak{e}}.$ 
For Lie groups containing a closed subgroup isomorphic to one of the aforementioned subgroups, the exponential map is not injective. The exponential map could be  surjective although it may even not be a local diffeomorphism, and may have a cut locus. This is the case for $SU(2).$
We announce the following result whose proof will be published elsewhere and which will be needed for the proof of Theorem \ref{theo:2-nilpotent}. It allows to directly get the value $\mu_\sigma$ of a Cartan-Schouten metric $\mu$ at any point $\sigma\in G$ from the sole value $\mu_\epsilon$ at the unit $\epsilon,$ exactly like for invariant metrics.

\begin{theorem}\label{theo:exp-diffeo} Let $G$ be a Lie group, $\G$ its Lie algebra, $\epsilon$ its unit. If $\exp: \G\to G$ is a diffeomorphism and $G$ has a Cartan-Schouten metric $\mu$, then setting $\bar \mu:=\mu_\epsilon$ and $\log:=\exp^{-1},$  we have, for every left invariant vector fields $x^+,$ $y^+$ on $G$ and $\sigma\in G,$
\beqn\label{eq:exp-diffeo} (\mu (x^+,y^+))(\sigma)&=&\sum\limits_{p,q=0}^{\infty} \frac{1}{p!q!2^{p+q} }   \bar\mu \Big(ad_{\log\sigma}^p\; x\; , \;ad_{\log\sigma}^q\;  y\Big).
\eeqn
 Conversely, if  $\exp: \G\to G$ is a diffeomorphism, then every metric  $\bar \mu$ on $\mathcal G$ satisfying (\ref{eq:Cartan-Schouten})  gives rise to a Cartan-Schouten metric  $\mu$ on $G$, via (\ref{eq:exp-diffeo}).
\end{theorem}

\subsection{Cartan-Schouten metrics on 2-nilpotent Lie groups \label{sect:2-nilpotent}}

Recall that a Lie group $G$ is said to be 2-step nilpotent if its Lie algebra, say $\mathcal G,$ is 2-step nilpotent (2-nilpotent, for short). Equivalently, the derived  ideal $ [\mathcal G,\mathcal G]$ is contained in the center $\mathfrak{Z}$ of $\G.$ 
We consider here Lie algebras $\G$ which, as vector spaces, split as  a direct sum $\mathcal G=V\oplus \mathfrak{Z}$ of two subspaces $V$ and $\mathfrak{Z},$ such that $[\mathcal G,\mathcal G]=[\mathcal G,V]=[V,V]= \mathfrak{Z}$ and  $[\mathcal G, \mathfrak{Z}]={\mathbf 0}.$ Every nondecomposable 2-nilpotent Lie algebra  lies in that category.  The 2-nilpotent Lie groups are the nonabelian Lie groups that are the closest possible to being  Abelian. In that regard, 
we expect them to be, in applications to Information Geometry, Statistics, Machine Learning, amongst the most important and handiest Lie groups with Cartan-Schouten metrics. On the other hand, they enjoy a very rich geometry and play special important roles in many areas  of Mathematics \cite{eberleinbook, eberleinENS,eberleinTransAMS}. 
Carnot groups of step $2$, H-type Lie groups, Heisenberg groups, are special  cases of   2-nilpotent Lie groups 
\cite{garcia, godoy, kumariKNN}. 
\begin{theorem}\label{theo:SetOfCartan-SchoutenMetrics}On a connected 2-nilpotent Lie group of dimension $n$, the set of Cartan-Schouten metrics is a connected and simply connected manifold of dimension $\frac{1}{2}n(n+1).$ 
\end{theorem}

In Theorem \ref{theo:Riemann-Cartan-Schouten}, we have shown that   2-nilpotent Lie groups are the only non-Abelian solvable Lie groups with 
at least one positive (or negative) definite Cartan-Schouten metric. Theorem \ref{theo:2-nilpotent} shows that a 2-nilpotent Lie group has infinitely many Cartan-Schouten of any desired signature $(p,n-p),$ for any integer $0\leq p\le n.$ 

\begin{theorem}\label{theo:2-nilpotent}Let $G$ be a 2-nilpotent Lie group, $\mathcal G$ its Lie algebra, $\dim G=n$. There is a 1-1 correspondance between the set of  left invariant metrics of signature $(p,n-p)$ and the set of Cartan Schouten metrics of signature $(p,n-p)$ on $G,$ for any integer $p,$ with $0\leq p\le n.$  
Namely, every metric  $\bar \mu$ on $\mathcal G$ gives rise to a flat Cartan-Schouten metric  $\mu$ of the same signature on any connected Lie group with Lie algebra $\mathcal G$. In particular, on the corresponding connected and simply connected Lie group $\tilde G$, $\mu$ is given by
 \beqn \label{eq:Cartan-chouten-2-nilpotent}(\mu (x^+,y^+))(\sigma)&=&\bar \mu (x ,  y) +\frac{1}{2}    \bar \mu \Big([\log\sigma, x] ,  y\Big)+
\frac{1}{2}    \bar \mu \Big( x , [\log\sigma,y]\Big) 
 \nonumber\\
&& + 
\frac{1}{4}    \bar \mu \Big([\log\sigma, x]\; ,\; [\log\sigma,y]\Big) ,
\eeqn
  for every $\sigma\in \tilde G$ and $x,y\in\G,$ where $\log  : \tilde G\to \mathcal G,$   $\log:=\exp^{-1}.$
\end{theorem}
\begin{proof}This is a direct consequence of more general results (see   Theorem \ref{theo:Cartan-Schouten-formula-Lie-Algebra}  and  Theorem \ref{theo:exp-diffeo}) whose proofs will be published elsewhere.  Every left invariant metrics $\bar \mu^+$ of signature $(p,n-p)$ on G is uniquely given by its value $\bar\mu$ at the unit $\epsilon .$
Since $G$ is 2-nilpotent, every metric $\bar\mu$ in $\G$ satisfies (\ref{eq:Cartan-Schouten}) in  Theorem \ref{theo:Cartan-Schouten-formula-Lie-Algebra}.
 We apply  Formula (\ref{eq:exp-diffeo}) of Theorem \ref{theo:exp-diffeo} to get (\ref{eq:Cartan-chouten-2-nilpotent}) which defines a unique Cartan-Schouten metric  of signature $(p,n-p)$ on $G.$ Conversely, if $\mu$ is a Cartan-Schouten metric on $G$, its value $\bar\mu$ at the unit gives rise to a unique  left invariant metric $\bar \mu^+$  on $G$ which, by construction, has the same signature as $\mu.$   
\end{proof}
One also has the following
\blem\cite{bi-invariant-and-noninvariant-metrics-ghanam-hindeleh-thompson}
Let $D\in\mathfrak{gl}(n,\mathbb R).$ Consider the semidirect sum $\mathbb RD\ltimes \mathbb R^n.$ A Lie group with Lie algebra $\mathbb RD\ltimes \mathbb R^n$  has a Cartan-Schouten metric if and only if it is 2-step nilpotent, or equivalently $D^2=0.$
\elem
So in particular, the group $O(2,\mathbb R)\ltimes \mathbb R^2$ of rigid displacements of the Euclidean plane, does not have any Cartan-Schouten metric.

\subsubsection{Proof of  Theorem \ref{theo:SetOfCartan-SchoutenMetrics}}
Let $G$ be a $2$-nilpotent Lie group and $\G$ its Lie algebra. We use Theorem  \ref{theo:2-nilpotent} to build  a biunivoque correspondence between  the set of left invariant metrics  and that of Cartan-Schouten metrics on $G.$ A left invariant metric and a  Cartan-Schouten metric are both uniquely given by their respective value at the unit $\epsilon$ of $G.$  So the correspondence maps any left invariant metric to the unique Cartan-Schouten metric which coincides with it at $\epsilon,$ and vice versa.  Thus, the space of 
Cartan-Schouten metrics  on a 2-nilpotent Lie group of                dimenion n, can be identified      with the set   of nonsingular symmetric matrices        in    $\mathfrak{gl}(n, \mathbb R),$  which    is       a smooth manifold of dimension $\frac{n(n+1)}{2}.$ 

\subsubsection{The Heisenberg  Lie group $\mathbb H_{2n+1}$\label{Heisenberg}}
Consider the $(2n+1)$-dimensional 
 Heisenberg group  $\mathbb H_{2n+1}:=\Big\{ \sigma=
\begin{pmatrix}1&x&z
\\ 0&1&y^T\\ 
0&0&1
\end{pmatrix}, \; x,\; y \in\mathbb R^n,\; z\in\mathbb R\Big\}$. Its Lie algebra $\mathcal{H}_{2n+1}$ is spanned by the $(n+2)\times (n+2)$ elementary matrices  $e_j:=E_{1,j+1},$ $e_{n+j}:=E_{j+1,n+2},$ $e_{2n+1}:=E_{1,n+2},$ $j=1,\dots,n.$  So the Lie bracket reads $[e_j,e_{n+j}]=e_{2n+1}.$ 
The fact that the exponential map is a diffeomorphism  can be written explicitly as $\sigma=\exp\Big(\sum\limits_{i=1}^n(x_ie_i+y_ie_{n+i})+(z-\frac{1}{2}xy^T)e_{2n+1}\Big),$ so $\log\sigma=\sum\limits_{i=1}^n(x_ie_i+y_ie_{n+i})+(z-\frac{1}{2}xy^T)e_{2n+1}.$
We identify $\mathbb H_{2n+1}$ with $\mathbb R^{2n+1}$, with the multiplication 
 $(x,y,z) (x',y',z')=(x+x',y+y',z+z'+xy'^T).$ 
 
The left-invariant vector fields corresponding to $e_j$, $e_{n+j},$ 
$e_{2n+1},$ are respectively 
\beqn
e_j^+=\frac{\partial}{\partial_{x_j}}, \quad  
e_{n+j}^+=\frac{\partial}{\partial_{y_j}}+x_j\frac{\partial}{\partial_z}, \quad  
e_{2n+1}^+=\frac{\partial}{\partial_z}, 
\eeqn
so we have $\frac{\partial}{\partial_{y_j}}=e_{n+j}^+-x_je_{2n+1}^+,$  whereas
 the left-invariant $1$-forms associated to the dual basis  $e_j^*,$ $e_{n+j}^*,$ $e_{2n+1}^*,$  
are:
$
(e_j^*)^+=dx_j, \quad 
(e_{n+j}^*)^+=dy_j,\quad 
(e_{2n+1}^*)^+=dz-x_jdy_j,
$ 
 $j=1,\dots,n.$ 
The Heisenberg Lie group does not have any biinvariant metric, since $[\mathcal H_{2n+1}, \mathcal H_{2n+1}]=\mathbb R e_{2n+1}$ is also the center of $\mathcal H_{2n+1}.$   
However, it does possess infinitely many  Cartan-Schouten metrics $\mu$ of any desired signature. 
On applying Formula (\ref{eq:Cartan-chouten-2-nilpotent}) of Theorem \ref{theo:2-nilpotent}, we explicitly give the expression of any such  metrics $\mu$ on $\mathbb{H}_{2n+1}.$  
We denote the constants $\bar\mu(e_p,e_q)$  by  $k_{p,q},$ $p,q=1,\dots,2n+1,$ where again $\bar \mu:=\mu_\epsilon.$  Next we give the explicit expression  $[\log\sigma,x],$ for every $\sigma\in\mathbb H_{2n+1}$ and $x$ in the basis ($e_1,\dots,e_{2n+1}$) of $\mathcal H_{2n+1}:$
\beqn[\log\sigma,e_i]&=&\sum\limits_{k=1}^n[(x_ke_k+y_ke_{n+k})+(z-\frac{1}{2}xy^T)e_{2n+1},e_i]=-y_ie_{2n+1},
\nonumber
\\ 
\;[\log\sigma,e_{n+i}]&=&\sum\limits_{k=1}^n[(x_ke_k+y_ke_{n+k})+(z-\frac{1}{2}xy^T)e_{2n+1},e_{n+i}]=x_ie_{2n+1}.
\eeqn
So we now get:

$(\mu(e_{2n+1}^+,e_{2n+1}^+))(\sigma)=\bar \mu (e_{2n+1} ,  e_{2n+1})=k_{2n+1,2n+1},$

$(\mu(e_{2n+1}^+,e_i^+))(\sigma)=\bar \mu (e_{2n+1} ,  e_i) +\frac{1}{2}    \bar \mu \Big([\log\sigma, e_i] ,  e_{2n+1}\Big) =k_{i,2n+1}-\frac{1}{2}k_{2n+1,2n+1} y_i,$

$(\mu(e_{2n+1}^+,e_{n+i}^+))(\sigma)=\bar \mu (e_{2n+1} ,  e_{n+i}) +\frac{1}{2}    \bar \mu \Big([\log\sigma, e_{n+i}] ,  e_{2n+1}\Big) $$=k_{n+i,2n+1}+\frac{1}{2}k_{2n+1,2n+1} x_i,$
In the same way, we obtain
\beqn (\mu(e_i^+,e_j^+))(\sigma)&=&\bar \mu (e_i ,  e_j) -\frac{1}{2}y_i    \bar \mu \Big(e_{2n+1} ,  e_j\Big)
- 
\frac{1}{2}  y_j  \bar \mu \Big( e_i , e_{2n+1} \Big) 
+ 
\frac{1}{4}  y_i y_j  \bar \mu \Big( e_{2n+1} , e_{2n+1} \Big) 
\nonumber\\
&=&\frac{1}{4} k_{2n+1,2n+1} y_i y_j   -\frac{1}{2}k_{j,2n+1}y_i -\frac{1}{2}k_{i,2n+1}  y_j +k_{i,j},\nonumber\\
(\mu(e_i^+,e_{n+j}^+))(\sigma)&=&\bar \mu (e_i ,  e_{n+j}) -\frac{1}{2}y_i    \bar \mu \Big(e_{2n+1} ,  e_{n+j}\Big) 
+ 
\frac{1}{2}  x_j  \bar \mu \Big( e_i , e_{2n+1} \Big) 
\nonumber\\&&
- 
\frac{1}{4}  y_i x_j  \bar \mu \Big( e_{2n+1} , e_{2n+1} \Big) \nonumber\\
&=&-\frac{1}{4} k_{2n+1,2n+1} y_i x_j   -\frac{1}{2}k_{n+j,2n+1}y_i +\frac{1}{2}k_{i,2n+1}  x_j +k_{i,n+j}
\eeqn
Similarly,         we compute   the         following:
\beqn (\mu(e_{n+i}^+,e_{n+j}^+))(\sigma)&=&\bar \mu (e_{n+i} ,  e_{n+j}) +\frac{1}{2}x_i    \bar \mu \Big(e_{2n+1} ,  e_{n+j}\Big)
+ 
\frac{1}{2}  x_j  \bar \mu \Big( e_{n+i} , e_{2n+1} \Big) 
\nonumber\\&&
+ 
\frac{1}{4}  x_i x_j  \bar \mu \Big( e_{2n+1} , e_{2n+1} \Big) \nonumber\\
&=&\frac{1}{4} k_{2n+1,2n+1} x_i x_j   +\frac{1}{2}k_{n+j,2n+1}x_i +\frac{1}{2}k_{n+i,2n+1}  x_j +k_{n+i,n+j} .
\eeqn

We 
 use the above to deduce the following
\beqn
 \mu(\frac{\partial}{\partial x_i},\frac{\partial}{\partial x_j})&=& \mu(e_i^+,e_j^+) ,\; 
 \mu(\frac{\partial}{\partial x_i}, \frac{\partial}{\partial x_{n+j}})=\mu(e_i^+,e_{n+j}^+) - x_{j}  \mu(e_i^+,e_{2n+1}^+) ,
\nonumber\\
 \mu(\frac{\partial}{\partial x_i},\frac{\partial}{\partial x_{2n+1}}) &=&\mu(e_i^+,e_{2n+1}^+) = k_{i,2n+1} -  \frac{1}{2 }  y_ik_{2n+1,2n+1}, 
\nonumber\\
\mu(\frac{\partial}{\partial x_{n+i}}, \frac{\partial}{\partial x_{n+j}}) &=&
\mu(e_{n+i}^+,e_{n+j}^+) 
-x_{j}            \mu( e_{n+i}^+,e_{2n+1}^+)+x_{i}x_{j}  \mu(e_{2n+1}^+, e_{2n+1}^+)
\nonumber\\
&& - x_{i}   \mu( e_{n+j}^+,e_{2n+1}^+)
,
\nonumber\\
\mu(\frac{\partial}{\partial x_{n+j}},\frac{\partial}{\partial x_{2n+1}}) &=&\mu( e_{n+j}^+,e_{2n+1}^+)-x_{j}\mu(e_{2n+1}^+, e_{2n+1}^+),\;
\nonumber\\
\mu(\frac{\partial}{\partial x_{2n+1}},\frac{\partial}{\partial x_{2n+1}})&=& \mu(e_{2n+1}^+, e_{2n+1}^+)  
.
\eeqn

Plugging the $f_{p,q}$ in, we finally get the following coefficients
\beqn
 \mu(\frac{\partial}{\partial x_i},\frac{\partial}{\partial x_j})&=& k_{i,j} -  \frac{1}{2 }  y_ik_{2n+1 ,j}  - \frac{1}{2 } 
 y_j k_{i ,2n+1} +
 \frac{1}{4}   y_i  y_j k_{2n+1,2n+1},
\nonumber
\eeqn
\beqn
 \mu(\frac{\partial}{\partial x_i}, \frac{\partial}{\partial x_{n+j}})
&=&   k_{i,n+j} -  \frac{1}{2 } y_ik_{n+j,2n+1}  -  \frac{1}{2 }
 x_j k_{i ,2n+1} +
 \frac{1}{4}   y_i  x_j k_{2n+1,2n+1}\; ,
\nonumber
\eeqn
\beqn
 \mu(\frac{\partial}{\partial x_i},\frac{\partial}{\partial x_{2n+1}}) &=& 
 k_{i,2n+1} -  \frac{1}{2 }  y_ik_{2n+1,2n+1}, \nonumber
\eeqn
\beqn
\mu(\frac{\partial}{\partial x_{n+i}}, \frac{\partial}{\partial x_{n+j}}) 
&=&
  k_{n+i,n+j} -  \frac{1}{2 } ( x_ik_{n+j,2n+1} +
 x_j k_{n+i ,2n+1}) 
+
 \frac{1}{4}   x_i  x_jk _{2n+1,2n+1},
\nonumber
\eeqn
\beqn
\mu(\frac{\partial}{\partial x_{n+j}},\frac{\partial}{\partial x_{2n+1}}) 
&=&  k_{n+j,2n+1} -  \frac{1}{2 }  x_jk_{2n+1,2n+1} ,\;
\mu(\frac{\partial}{\partial x_{2n+1}},\frac{\partial}{\partial x_{2n+1}}) 
= 
k_{2n+1,2n+1}.\nonumber
\eeqn

In  summary, we have  the following.
\begin{proposition}Any Cartan-Schouten metric on $\mathbb H_{2n+1}$ is of the following form
\beqn
\mu&=&\sum\limits_{i,j=1}^n\Big[ \Big(k_{i,j} -  \frac{1}{2 }  y_ik_{2n+1 ,j} -  \frac{1}{2 }
 y_j k_{i ,2n+1} +
 \frac{1}{4}   y_i  y_j k_{2n+1,2n+1}
\Big)\;dx_idx_j\nonumber\\
&&+
\Big( k_{i,n+j} -  \frac{1}{2} y_ik_{n+j,2n+1} -  \frac{1}{2}
 x_j k_{i ,2n+1} +
 \frac{1}{4}   y_i  x_j k_{2n+1,2n+1}\Big)\;dx_{i}dx_{n+j}
\nonumber\\
&&+
\Big( k_{n+i,n+j} -  \frac{1}{2 }  x_ik_{n+j,2n+1} -  \frac{1}{2 }
 x_j k_{n+i ,2n+1} 
+
 \frac{1}{4}   x_i  x_jk _{2n+1,2n+1}\Big)\;dx_{n+i}dx_{n+j}\Big]
\nonumber\\
&&
+\sum\limits_{j=1}^n\Big[ \Big(k_{j,2n+1} -  \frac{1}{2 }  y_jk_{2n+1,2n+1}\Big) dx_{j}
+\Big( k_{n+j,2n+1} -  \frac{1}{2 }  x_jk_{2n+1,2n+1}\Big)                                   dx_{n+j}\Big]dz
\nonumber\\
&&+ k_{2n+1,2n+1}\;dz^2,
\eeqn
where the $k_{p,q}$'s  are (constant) real parameters,   $p,q=1,\dots,2n+1.$
\end{proposition}

\subsubsection{Cartan-Schouten metrics on Carnot groups}
Here, we consider Carnot groups of order 2 and their Lie algebras, namely the so-called on H-type  Lie groups and Lie algebras.
Let  $G$ be a simply connected  Lie group whose Lie algebra $\mathcal G$
is graded $\mathcal G=V\oplus \mathcal Z$ such that $[\mathcal G,\mathcal G]= [V,V] =\mathcal Z$ and $[\mathcal G,\mathcal Z]= 0.$ One endows $\mathcal G$ with an inner product $ \langle,\rangle$ such that the orthogonal $V^\bot$ of $V$ is $\mathcal Z.$  This induces, for each $z\in \mathcal Z, $ a linear map 
$J_Z: V\to V,$ 
given by $\langle J_Z(w),w'\rangle:= \langle Z,[w,w']\rangle$ for every $w, w'\in V.$ 
 The map $J_Z$ is skew-symmetric or equivalently $\langle J_Z(w),w'\rangle=-\langle w,J_Z(w')\rangle$.
The Lie algebra $\mathcal G$ is said to be of H-type if for every $v$ in $V$ with unit length, $ad_v$ is an isometry from $\Big(\ker(ad_v)\Big)^\bot$ to $\mathcal Z.$
 The endomorphisms $J_Z$ satisfy  $J_Z J_{Z'}+ J_{Z'}J_Z= -2 \langle Z,Z\rangle \mathbb I_V,$  in particular $J_Z^2=-\langle Z,Z\rangle\mathbb I_V,$ where $\mathbb I_V: V\to V$ is the identity map of $V.$
If $(Z_1,\cdots,Z_m)$ is an orthonormal basis of $\mathcal Z$, that is, $ \langle Z_i,Z_j\rangle=\delta_{i,j},$ then $[v,w]=\displaystyle\sum_{i=1}^m\langle J_{Z_i}(v),w\rangle Z_i$ for any $v,w\in V.$
We consider an orthogonal basis $(X_1,\cdots, X_n, Z_1,\cdots,Z_m)$ of $\mathcal G$ such that  $(X_1,\cdots, X_n)$ is a basis of $V$ and $( Z_1,\cdots,Z_m)$ a basis of $\mathcal Z$ such that 
$\langle Z_i,  Z_i\rangle =1$ for every $i=1,\cdots, m$.
The Lie bracket reads                                             	$[X_i,X_j]=\displaystyle\sum_{k=1}^m\langle J_{Z_k}(X_i),X_j\rangle Z_k.$
As a manifold, $G$ is identified with $\mathbb R^{n+m}$. We denote elements $\sigma=\exp(x_1X_1+\cdots+x_n X_n+z_1Z_1 + \cdots +z_mZ_m)$ of $G$, by $(x,z)$ where $x=(x_1,x_2,\cdots,x_n)\in\mathbb R^n$ and $z=(z_1,z_2,\cdots,z_m)\in\mathbb R^m$, correspond to normal coordinates associated with $\exp: \G\to G$, so  the product  reads
\begin{eqnarray}\label{eq:H-type}
(x,z)(x',z') &=&(x+x',z+z'+\frac{1}{2}\langle \gamma x,x'\rangle )\nonumber\\
&=&(x_1+x'_1,\cdots, x_n+x'_n, z_1+z_1'+\frac{1}{2}\langle \gamma^1 x,x'\rangle, \cdots, 
\nonumber\\
&&z_m+z_m'+\frac{1}{2} \langle \gamma^n x,x'\rangle)
\end{eqnarray}
where 
$\langle \gamma^k x,x'\rangle=  \displaystyle \sum_{i=1}^n \sum_{l=1}^n\gamma_{il}^k x_lx_i'$ and $\gamma^k=(\gamma_{il}^k)$ is the matrix of the linear map $J_{Z_k}.$ Note that we have the equality 
\beqn
\gamma_{jl}^k=-C_{jl}^{n+k}\;,\; \;\text{  with  } [X_j,X_l]=\sum\limits_{k=1}^m C_{jl}^{n+k}Z_k.
\eeqn

The left invariant vector fields $X_j^+$, $j=1,\cdots,n$  corresponding to $X_j$ are 
\begin{eqnarray}X_j^+=\frac{\partial}{\partial x_j}+\frac{1}{2}\displaystyle \sum_{k=1}^m \Big( \langle Z_k, [X,X_j]\rangle\Big)\frac{\partial}{\partial z_k} = \frac{\partial}{\partial x_j} - \frac{1}{2}\displaystyle \sum_{k=1}^m \Big( \langle J_{Z_k} (X_j), X\rangle\Big)\frac{\partial}{\partial z_k}\end{eqnarray}
where $X= \frac{1}{2}\displaystyle \sum_{k=1}^n x_jX_j.$ In other words, 
\begin{eqnarray}X_j^+= \frac{\partial}{\partial x_j} + \frac{1}{2}\displaystyle \sum_{k=1}^m \sum_{l=1}^n\gamma_{jl}^k x_l \frac{\partial}{\partial z_k}.\end{eqnarray}

On the other hand, for $j=1,\cdots,m,$ we have 

\begin{eqnarray}
Z_j^+=\frac{\partial}{\partial z_j}.
\end{eqnarray}
Let be  $(X_1^*,\cdots, X_n^*, Z_1^*,\cdots,Z_m^*)$ the dual basis of  $(X_1,\cdots, X_n, Z_1,\cdots,Z_m)$,  with corresponding left invariant 1-forms

\begin{eqnarray}
\eta_j=(X_j^*)^+=dx_j, \; \; j=1,\cdots, n
\end{eqnarray}
and
\begin{eqnarray}
 \nu_j=(Z_j^*)^+=dz_j + \frac{1}{2}\displaystyle \sum_{k=1}^n \Big( \langle J_{Z_j} (X_k), X\rangle\Big)dx_k =dz_j-   \frac{1}{2}\displaystyle \sum_{i=1}^n \sum_{l=1}^n\gamma_{il}^j x_l dx_i\;.
\end{eqnarray}

To  look     for Cartan-Schouten metrics,  we  apply Formula (\ref{eq:Cartan-chouten-2-nilpotent}) of Theorem \ref{theo:2-nilpotent},  and we   follow the steps  as in Example \ref{Heisenberg}.                   We summarize our results as follows. 
\bthm\label{thm:general-compatible-metric-on-h-type-group}
Let $G$ be an  $(n+ m)$-dimensional H-type Lie group  whose Lie algebra 
$\mathcal G$ decomposes as $\mathcal G=V\oplus \mathcal Z$ such that 
$[\mathcal G,\mathcal G]= [V,V] =\mathcal Z$  and $[\mathcal G,\mathcal Z]= 0$,
where $\dim V=n$ and $\dim \mathcal Z=m$. 
Any  Cartan-Schouten metric $\mu$ on $G$ is of the form: 
\beq \label{eq:metric-form-2}
\mu&=&\sum_{i=1}^n\left[\frac{1}{4}\sum_{\alpha,\beta=1}^{m}\sum_{p,l=1}^{n}
C_{pi}^{n+\alpha}C_{li}^{n+\beta}  d_{n+\alpha,n+\beta}x^px^l
+\sum_{\alpha=1}^{m}\sum_{l=1}^n\gamma_{il}^{\alpha}d_{i,n+\alpha} x^l+d_{ii}\right](dx^i)^2 \cr 
& &+\sum_{\alpha=1}^m d_{n+\alpha,n+\alpha}\left(dz^\alpha- 
 \frac{1}{2} \sum_{l,p=1}^n\gamma_{pl}^\alpha x^l dx^p \right)^2 
\cr 
  & &
 + \sum_{1\le i< j\le n}\left[
\frac{1}{4}\sum_{\alpha,\beta=1}^{m}\sum_{p,l=1}^{n}C_{pi
}^{n+\alpha}C_{lj}^{n+\beta}  d_{n+\alpha,n+\beta}x^px^l \right. \cr 
& & \left. +\frac{1}{2}\sum_{\alpha=1}^{m}\sum_{l=1}^n\Big(C_{li}^{n+\alpha}d_{j,n+\alpha}+C_{lj}^{n+\alpha}d_{i,n+\alpha}
\Big) x^l+d_{ij}\right]dx^i dx^j \cr 
& & + \sum_{i=1}^n\sum_{\alpha=1}^m\left[
\frac{1}{2}\sum_{\beta=1}^{m}\sum_{k=1}^nC_{ki}^{n+\beta}d_{n+\alpha,n+\beta}x^k+d_{i,n+\alpha}\right]dx^i dz^\alpha 
 \cr 
& & 
\left.
- \frac{1}{2}\sum_{i,l,p=1}^n\sum_{j=1}^m\left[
\frac{1}{2}\sum_{l=n+1}^{n+m}\sum_{k=1}^nC_{ki}^{n+l}d_{jl}x_k+d_{i,n+j}
\right]\gamma_{pl}^j x^l dx^idx^p \right.
 \cr
  & & 
 + \sum_{1\le \alpha< \beta \le m}d_{n+\alpha,n+\beta}
\Big(dz^\alpha- 
 \frac{1}{2}\sum_{p=1}^n \sum_{l=1}^n\gamma_{pl}^\alpha x^l dx^p \Big)\Big(dz^\beta- 
 \frac{1}{2}\sum_{p=1}^n \sum_{l=1}^n\gamma_{pl}^\beta x^l dx^p \Big),
\eeq 
where the $d_{\alpha\beta}$'s are real constants numbers such that the determinant of the 
matrix $(d_{\alpha\beta})_{\alpha,\beta=1,\dots,m+n}$ is nonzero, and $\gamma_{ij}^k=C_{ij}^{n+k}=-C_{ji}^{n+k}$      are real    numbers considered as parameters. 
\ethm
Note that  each parameter $\gamma=(\gamma_{ij}^k)$ defines  a 2-nilpotent Lie algebra and Lie group of H-type, as in (\ref{eq:H-type}). All the parameters    $(\gamma_{ij}^k)$  belong to the manifold of 2-nilpotent Lie algebras, which in turn, sits inside  that of nilpotent Lie    algebras. Fixing
 $(\gamma_{ij}^k)$    and varying $(d_{\alpha\beta})_{\alpha,\beta=1,\dots,m+n}$ amounts   to fixing the  Lie algebra and looking for all the Cartan-Schouten metrics on it. Note also that the coefficients of the metrics are degree 2 polynomials.
\subsubsection{Biinvariant metrics on 2-nilpotent  Lie groups}
We summarize some of the results gathered from \cite{benito-2-step-nilp-quadratic,Revoy-3forms} as follows.
\begin{proposition}\cite{benito-2-step-nilp-quadratic,Revoy-3forms}Let  $\G $ be a  2-nilpotent Lie algebra with center $\mathfrak{Z}$ satisfying $[\mathcal G,\mathcal G]=\mathfrak{Z}$. Suppose $\mathcal G$ has an ad-invariant metric, say $\mu$.  Then  $\mathcal G$  must be even dimensional. Moreover, there exists a subspace $V$ such that, ${\mathfrak Z}$ and $V$ are Lagrangian subspaces of $\mathcal G$  in duality with respect to $\mu$.
More precisely, we have the following,
\begin{eqnarray}
\mu(\mathfrak{Z},\mathfrak{Z} ) &=&\mu(V,V)=0, \;\;
 2\dim V =2\dim\mathfrak{Z} = \dim\mathcal G, \label{eq:lagrangian1}\\
\mathcal G&=& V\oplus \mathfrak{Z} \text{ (direct sum of vector spaces}). \label{eq:lagrangian2}
\end{eqnarray}

\end{proposition}

So, every ad-invariant metric on a 2-nilpotent Lie algebra has signature $(n,n).$ We get the following more precise result.

\begin{proposition} Let  $\G $ be a  2-nilpotent Lie algebra of dimensions $2n,$ with center $\mathfrak{Z}$ satisfying $[\mathcal G,\mathcal G]=\mathfrak{Z}$. Suppose $\mathcal G$ has an ad-invariant metric, say $\mu$.  Then we can choose a subspace $V$ complementary to $\mathfrak{Z}$  and a basis $(e_1,\dots,e_{2n})$ of $\G$, where $(e_1,\dots,e_n)$ and $(e_{n+1},\dots,e_{2n})$ are bases of $V$ and $\mathfrak{Z}$ respectively, such  that $\mu(e_i,e_j)=\mu(e_{n+i},e_{n+j})=0$ and $\mu(e_i,e_{n+j})=\delta_{i,j},$ $i,j=1,\dots,n.$
The Lie bracket of $\G$ reads $[e_i,e_j]=\displaystyle\sum_{k=1}^nC_{ij}^{n+k}e_{n+k}$,  where the structure constants $C_{ij}^{n+k}$ satisfy the following identities:
  \begin{eqnarray}
 C_{ij}^{n+k} &=&   C_{jk}^{n+i} = - C_{ik}^{n+j}. \label{constant-structures}
 \end{eqnarray}
\end{proposition}
\begin{proof}
The first claim is a general fact which we prove by induction on the dimension $2n$ of the vector space, say $E,$ underlying $\mathcal G$.  This is true when $n=1$. Indeed, suppose a $2$-dimensional vector space 
$E$ has a nondegenerate bilinear symmetric form $\mu$ such that a line $\mathbb R v$ such that is $\mu(v,v)=0.$ Let $\tilde v$ be an element of $\mathcal{E}$ satisfying $\mu(v,\tilde v)\neq 0$. By setting 
\begin{eqnarray}\label{dual1}
v^*:=\frac{1}{\mu(v,\tilde v)} \Big(\tilde v-\frac{\mu(\tilde v,\tilde v)}{2\mu(v,\tilde v)}\; v \Big),
\end{eqnarray} 
we get a basis $(v,v^*)$ of  $E$  satisfying 
\begin{eqnarray}\label{dual1b}\mu(v,v)=\mu(v^*,v^*)=0, \; \text{ and } \; \mu(v,v^*)=1.\end{eqnarray} 
Now suppose this property true up to order $n-1$, for some $n\ge 2$ and let us show that it also holds true for $n.$ 
Let $E$ be a $2n$-dimensional vector space endowed with a nondegenerate bilinear symmetric form $\mu$ such that an $n$-dimensional subspace $\tilde E$ is totally isotropic with respect to $\mu.$ Let $v$ be a nonzero vector of $E$ not belonging to $\tilde E.$ The linear form $f_v$ defined on $\tilde E$  by $f_v(x)=\mu(v,x)$ for any $x\in \tilde E,$ is nonzero and $\dim\ker f_v = n-1.$ So we write $\tilde E$ as $\tilde E=\mathbb R \tilde v \oplus \ker(f_v)$ for some $\tilde v \in \tilde E$ satisfying $0\neq f_v(\tilde v) =\mu(v,\tilde v). $ In the 2-dimensional nondegenerate vector space $\mathcal E:=\mathbb Rv\oplus \mathbb R\tilde v$ we apply (\ref{dual1}) to get a basis 
$\Big(e_{n}:= \frac{1}{\mu(\tilde v,v)} (v-\frac{\mu(v,v)}{2\mu(\tilde v,v)}\tilde v), e_{2n}:=\tilde{v}\Big)$ satisfying (\ref{dual1b}). The orthogonal $\mathcal E^\bot$ of   $\mathcal E$ in $E$ with respect to $\mu$, is a $2(n-1)$-dimensional nondegenerate vector space that contains $\ker (f_v)$ as a totally isotropic $(n-1)$-dimensional subspace and we have the decomposition $E=\mathcal{E}\oplus \mathcal E^\bot$. By hypothesis on the dimension $2(n-1)$, there is a basis  $(e_1,\dots,e_{n-1}, e_{n+1},\dots,e_{2n-1})$   of $\mathcal E^\bot$ such that $( e_{n+1},\dots,e_{2n-1})$ is a basis of $\ker(f_v)$  and  $\mu(e_i,e_{j}) = \mu(e_{n+i},e_{n+j})=0$ and $\mu(e_i,e_{n+j})=\delta_{i,j}$, for every $i,j=1,\dots,n-1.$  
Altogether, we have a basis $(e_1,\dots,e_{2n})$ of $E$ satisfying  $\mu(e_i,e_{j}) = \mu(e_{n+i},e_{n+j})=0$ and $\mu(e_i,e_{n+j})=\delta_{i,j}$, for every $i,j=1,\dots,n.$ , such that $(e_{n+1},\dots,e_{2n})$ is a basis of the Lagrangian subspace $\tilde E.$

The second claim is proved as follows. From  the following equalities, due to the ad-invariance,  
\begin{eqnarray}\mu([e_i,e_j],e_{k}) &=&  \mu(e_i,[e_j,e_{k}]) =- \mu(e_j,[e_i,e_{k}]), \nonumber
 \end{eqnarray}
we get 
\begin{eqnarray}\sum\limits_{m=1}^nC_{ij}^{n+m}\mu(e_{n+m},e_{k}) &=&  \sum\limits_{p=1}^nC_{jk}^{n+p} \mu(e_i,e_{n+p}) =- \sum\limits_{q=1}^nC_{ik}^{n+q}\mu(e_j,e_{n+q}), 
 \end{eqnarray}
which simplifies to
\begin{eqnarray}
 C_{ij}^{n+k} &=&   C_{jk}^{n+i} = - C_{ik}^{n+j}. \label{constant-structures}
 \end{eqnarray}
\end{proof}

Note that the above equalities (\ref{constant-structures}) imply that  
 \begin{eqnarray}
 C_{ij}^{n+i} = C_{ij}^{n+j}= 0, \;\; i,j=1,\dots,n.
 \end{eqnarray}

\section{Dual connections and Statistical structures}\label{sect:biinvariant-dual-connections-statistics}

\subsection{Statistical structures}
Let $G$ be a Lie group with a Cartan-Schouten metric $\mu.$  For any totally symmetric covariant $3$-tensor $S$ on $G$, define the tensor $A$ as 
 \beqn\label{eq:3-tenor-connection} \mu(A(X,Y),Z)=S(X,Y,Z). \eeqn
From (\ref{eq:3-tenor-connection}), one extracts the corresponding dual connections $\nabla^1$ and $\nabla^{-1}$ to get 
$\nabla^{1}:=\nabla -\frac{1}{2} A$
 and $\nabla^{-1}:=\nabla +\frac{1}{2} A.$ That is, 
\beqn \nabla^1_XY:=\frac{1}{2}([X,Y] -A(X,Y)) ,\; \nabla^{-1}_XY:=\frac{1}{2}([X,Y] +A(X,Y)) ,
\eeqn
for any vector fields $X,Y$ on $G.$
It is readily seen that both $ \nabla^1$ and $ \nabla^{-1}$ are torsion-free, dual  with respect to $\mu$ and satisfy $\nabla^{1}\cdot \mu=-\nabla^{-1}\cdot \mu=S.$

However, since $\mu$ is not left (nor right) invariant in general, $A$ is not left invariant in general  if we choose $S$ to be left invariant.  
We will call a left invariant Cartan-Schouten statistical structure, a triplet $(G,\mu,\bar\nabla)$ where $G$ is a Lie group,  $\mu$ and $\bar\nabla$ are respectively a Cartan-Schouten metric and a left invariant connection on $G$ such that $\bar\nabla\mu$ is a totally symmetric $3$-tensor. From now on, $\bar\nabla$ is considered to be torsion free. If we write $\bar\nabla$ as $\bar\nabla=\nabla -\frac{1}{2} \mathfrak{t},$   with $\mathfrak{t}:\G\times\G\to\G$ symmetric. We set $\mu_\epsilon=:\bar\mu.$     Then  
the total symmetry identity  $\bar\nabla\mu(x^+,y^+,z^+)=\bar\nabla\mu(y^+,x^+,z^+)$ taken at the unit $\epsilon,$  is equivalent to     
\beqn \label{eq:hessian}\bar\mu(\mathfrak{t}(x,y),z)+\bar\mu(y,\mathfrak{t}(x,z))=\bar\mu(\mathfrak{t}(y,x),z)+\bar\mu(x,\mathfrak{t}(y,z)).
\eeqn

A Lie algebra together with  a metric and a (locally) flat torsion free connection satisfying (\ref{eq:hessian}), is called a Hessian Lie algebra (see e.g. \cite{aubert-medina,shima1980}).
We have proved the following.
\begin{proposition}\label{prop:left-invariant-Cartan-Schouten-statistics} There is a 1-1 correspondence between left invariant Hessian structures ($\bar\mu,\bar\nabla$)          where the metric  $\bar\mu$ satisfies (\ref{eq:Cartan-Schouten})  and flat statistical structures on Lie groups
associated to Cartan-Schouten metrics together with  left invariant connections.
 \end{proposition}

\subsection{Biinvariant dual connections}

As above, $\nabla$ stands for the Cartan-Schouten canonical connection. 
A connection on $\mathcal G$ is of the form 
$\bar \nabla_xy=\frac{1}{2}[x,y]-\frac{1}{2}\mathfrak{k}(x,y),$ where $\mathfrak{k}:\mathcal G\times \mathcal G\to \mathcal G$ is some bilinear map. Write $\mathfrak{k} = \mathfrak{k}^{skew}+\mathfrak{k}^{sym},$ where 
$\mathfrak{k}^{skew} (x,y)=\frac{1}{2}(\mathfrak{k}(x,y))-\mathfrak{k}(y,x)) $ and $\mathfrak{k}^{sym}(x,y)=\frac{1}{2}(\mathfrak{k}(x,y))+\mathfrak{k}(y,x)). $
 The torsion $T^{\bar \nabla}$ is given by $T^{\bar \nabla}=-\mathfrak{k}^{sym}.$  So the connection
$\bar \nabla$ is torsion-free if and only if 
$\mathfrak{k}$ is symmetric.
 If $\mathfrak{k}$ is skew-symmetric, then obviously,  the connections $\nabla$ and $\bar \nabla$ share the same geodesics.
The connection $\bar \nabla$ is biinvariant if and only if it satisfies the following equation (see e.g. \cite{pennec-biinvariant-means}, \cite{benayadi-boucetta-biinvariant-special-connections-Poisson}) 
\beqn \label{biinvariant1}
[z,\bar\nabla_xy ] = \bar\nabla_{[z,x]}y + \bar\nabla_x [z, y ]\;.\eeqn
We deduce that (\ref{biinvariant1}) is equivalent to the following
\beqn \label{biinvariant2}
 [ z, \mathfrak{k}(x,y)] = \mathfrak{k}([z, x], y) + \mathfrak{k}(x, [z, y]),
\eeqn
or another  equivalent statement is that the covariant derivative of $\mathfrak{k}$ vanishes, $\nabla \mathfrak{k}=0.$ 
 If $\mathfrak{k}$ is skew-symmetric, the biinvariance condition is also equivalent to saying that $\mathfrak{k}$ is a 2-cocycle for the adjoint representation of $(\mathcal G, [,]).$ That is, the Chevalley-Eilenberg differential of $\mathfrak{k}$ vanishes:
\beqn  \partial  \mathfrak{k}(x,y,z):=
 [ x, \mathfrak{k}(y,z)] - \mathfrak{k}([x, y],z) - \mathfrak{k}(y, [x, z]) =0,
\eeqn
 for any $x,y,z\in\G.$
In particular, for any linear map $\psi: \mathcal G \to \mathcal G,$ the  coboundary  $\mathfrak{k}_\psi=\partial \psi=2\nabla \psi,$ defined by $\mathfrak{k}_\psi(x,y)=[x,\psi (y)]-\psi([x,y])$, is such that $\nabla^\psi:=\nabla-\frac{1}{2}\mathfrak{k}_\psi$ is a biinvariant connection for which every 1-parameter subgroup through the unit of the Lie group, is a geodesic.

We    summarize the above in the following
\begin{proposition} \label{prop:biinvariantconnections}The is  a 1-1 correspondence between                           biinvariant Cartan-Schouten connections                                                                                                                                                                                                                                                                                                                                                                                                                                                                                                                                                                                                                                                                                                                            on a Lie         group and      the  second space cocyles $\mathfrak{k}:\wedge^2\G\to\G,$ for the adjoint action of $\G.$  
\end{proposition}
Any bilinear map $\bar {\mathfrak{k}}: \G/ [\G,\G]\times \G/ [\G,\G]\to Z(\G)$ naturally lift up to a biinvariant connection $\mathfrak{k}$ on $\G.$ Here the Lie algebra $\G/ [\G,\G]$ is the quotient of $\G$ and its derived ideal $\G/ [\G,\G,$ and $ Z(\G)$ is the center of $\G.$
 The Lie algebra $\G$ is $2$-nilpotent if and only if  $[\G,\G]\subset Z(\G).$ So the dimension of the space of biinvariant connections is greater than the dimension of the space  of  bilinear  maps  $\bar {\mathfrak{k}}: \G/ [\G,\G]\times \G/ [\G,\G]\to  [\G,\G].$

We look at the case where  $\mathfrak{k}$ is symmetric, so that  the connection on $\G$ given by $\bar \nabla_xy=\frac{1}{2}[x,y]-\frac{1}{2}\mathfrak{k}(x,y),$  is torsion free. We define the 3-tensor $S,$ with
\beqn S(x,y,z):=\mu(\mathfrak{k}(x,y),z).
\eeqn 
The total symmetry of $S$ is equivalent to the relation
\beqn\mu(\mathfrak{k}(x,y),z) = \mu(y,\mathfrak{k}(x,z)),
\eeqn 
for any $x,y,z\in\G.$ The following parameter family of biinvariant $\alpha$-connections
\beqn \nabla^\alpha_xy:=\frac{1}{2}([x,y] -\alpha\; \mathfrak{k}(x,y)) ,\; \nabla^{-\alpha}_xy:=\frac{1}{2}([x,y] +\alpha\; \mathfrak{k}(x,y)) 
\eeqn
is such that $\nabla^\alpha$ and $\nabla^{-\alpha}$ are dual with respect to $\mu,$ for any $\alpha\in\mathbb R.$ We also have
\beqn 
\bar \nabla^{\alpha} \mu(x^+,y^+,z^+) =  \frac{\alpha}{2}\Big(\mu(\mathfrak{k}(x^+,y^+),z^+)  +  \mu(y^+,\mathfrak{k}(x^+,z^+))\Big)\;.
\eeqn
So if $S$ is symmetric, then $\bar \nabla^{\alpha} \mu =\alpha \;S=-\bar \nabla^{-\alpha} \mu.$

Here  is a method for constructing biinvariant torsion free flat connections.
\begin{theorem}\label{theo:biinvariant-connections}
Let $G$ be a Lie group, $\G$ its Lie algebra, $Z(\G)\neq 0$ the center of $\G.$  Let $X_1,\dots,X_p\in Z(\G),$ where
$p\ge 1$ is an integer.  Let  $B_j:\G\times\G\to\mathbb R,$ be bilinear symmetric with    $\nabla B_j=0,$ $j=1,\dots,p.$                                                                                                                                                                                                                                                                                  
 The                                                         
map $\mathfrak{k}:\G\times\G\to\G,$ $\mathfrak{k}(x,y)=\sum\limits_{j=1}^pB_j(x,y)X_j,$ defines a biinvariant torsion-free connection $\bar\nabla=\nabla-\frac{1}{2}\mathfrak{k}.$                                                                                                                                                                                                                                                                                                                                                                                                                                                                                                                                                                                   
If $G$ is $2$-nilpotent and $B_j(x,X_k)=0,$ $j,k=1,\dots,p,$ for any $x\in\G,$ then  $\bar\nabla$ is flat.                      
In particular, for any closed 1-forms $f_j,,$ $j=1,\dots,p,$  set $B:=\sum\limits_{j=1}^p f_j\otimes f_j$ and  
$\mathfrak{k}(x,y):=B(x,y)X,$ for some $X\in Z(\G).$ Then $\bar\nabla=\nabla-\frac{1}{2}\mathfrak{k}$  is biinvariant, torsion free and flat.
\end{theorem}
\begin{proof}

Suppose $B_j$ are symmetric bilinear forms on $\mathcal G$  such that      $\nabla B_j=0,$ and $X_j\in Z(\mathcal G),$ $j=1,\dots,p.$ Consider the symmetric bilinear                                                           
map $\mathfrak{k}:\G\times\G\to\G,$ $\mathfrak{k}(x,y)=\sum\limits_{j=1}^pB_j(x,y)X_j.$    
By abuse of notation, we use here the same notations for both a quantity in the Lie algebra and the corresponding left invariant quantity. For example a left invariant vector field $x^+$ is simply denoted by $x.$
We use the fact that covariant derivative of each $B_j$ vanishes, that is,                             
\beqn\nabla B_j(x,y,z)=-\frac{1}{2}\Big(B_j([x, y],z)+B_j(y,[x, z])\Big)=0.\eeqn
  Since $[x_0,x]=0,$ for any $x\in\G,$ the equation which expresses the biinvariance reads                         
\beqn  [ x, \mathfrak{k}(y,z)] - \mathfrak{k}([x, y],z) - \mathfrak{k}(y, [x, z])  
&=&  \sum\limits_{j=1}^p\Big(B_j(y,z)[x,X_j]\nonumber\\
&&+\;2(\nabla B_j(x, y,z))X_j\Big) =0.
\eeqn

Thus     the  connection $\bar \nabla=\nabla-\frac{1}{2}\mathfrak{k}$ is biinvariant and, since $\mathfrak{t}$ is symmetric,      $\bar \nabla$ is also torsion-free.   
If  $B_j(x,X_k)=0,$ for any $x\in\G$ and $j,k=1,\dots,p,$  the curvature of $\bar \nabla$ coincides with            that of $\nabla,$
\beqn R(x,y)z=-\frac{1}{4}[[x,y],z].
\eeqn 
So if in addition, $\G$ is 2-nilpotent 
then $R(x,y)z=0,$ for  any $x,y,z\in\G.$                         
\end{proof}

\section{A new model for statistics, machine learning and data science}\label{sect:newmodel}
\subsection{On 2-nilpotent   Lie  group     structures  on $\mathbb  R^N$}

2-step nilpotent Lie groups are the only Lie groups for which the Cartan-Schouten connection is flat hence entailing that any  Cartan-Schouten metric is a Hessian metric. Thus their study benefits from the nice properties of Hessian   metrics. Moreover,                                                                      the fact that the Riemannian mean of any  Cartan-Schouten metric coincides  with  the  biinvariant means            (exponential  barycenter)  of the Lie group        as discussed in Section \ref{exponentialbarycenter}, makes them much more attractive. 

Here is our main point: seen as a manifold,  $\mathbb R^N$  is the (common) universal cover  for all nilpotent Lie groups,  
thus we can look at any phenomenon happening in $\mathbb R^N$ as a phenomenon sitting on a manifold home to infinitely many nilpotent Lie group structures. So we can look for those Lie group structures which are compatible with the phenomenon at hand, i.e. those Lie group structures for which our phenomenon is left or right invariant, biinvariant, or the group is a group of symmetries, etc. 
Along these lines, in the present paper we concentrate on the case of 2-step nilpotent Lie groups which offer many advantages,    with many foreseen applications.

\subsection{Exponential  barycenter                                           \label{exponentialbarycenter}}

The group exponential  barycenter  of a dataset ($\sigma_i$) is a solution $\mathfrak{m}$ (if it exists, see\cite{miolane-pennec2015})  of the following barycenter      equation                
 \begin{eqnarray}\label{eq:biinvariant-barycenter}\displaystyle\sum_{i=1}^p \log(\mathfrak{m}^{-1}\sigma_i)=0 ,\end{eqnarray}
where $\log$ is (locally) the inverse   of  the exponential  map.  In our case, the  exponential map of Cartan-Schouten metrics coincides with that of the Lie group. Thus,  the solution of (\ref{eq:biinvariant-barycenter}) coincides with the (biinvariant) mean of the Lie                                                                                group. 
Equation (\ref{eq:biinvariant-barycenter}) may not admit a unique solution, or may not even admit a solution at all.  We will rather use the terminology of cloud, instead of  dataset.  

\subsection{On a  new model of parametric means}
Applying the above study and discussions, we propose here a brand new model of parametric mean, say $\mathfrak{m}, $  for statistics, machine learning and data science.
Since the parameter evolves in the whole manifold of 2-nilpotent Lie algebras, one enjoys a wide room of parameters to manoeuvre, unlike the  traditional methods  such as the arithmetic mean,     median, mode,  expectation,  least square method, maximum likelihood, linear regression, ...  This is particularly suitable for fitting data or  estimating  the  parameters     amid several constraints. The space of application of this mean is the ordinary Euclidean space $\mathbb R^N$  or any vector space of dimension $N$, for any given integer $N\ge 2.$ We first partition $N$ into an arbitrary ($n,m$), where $n+m=N.$                                                                                
 For   a cloud of $p$ points $\{\sigma_i\}_{i=1,\dots,p}$  where we denote the   coordinates of each point $\sigma_i$ by $(x_i^1,\dots,x_i^{n+m})\in \mathbb R^{n+m},$ we  will let $E(x^r)$ stand for the arithmetic mean of the $r-$th components $E(x^r):=\frac{1}{p}\sum\limits_i^px_i^r,$ $r=1,\dots,n+m.$

\begin{theorem} Every cloud (dataset) of  points in $\mathbb R^{n+m}$ admits a biinvariant mean for some parameter  family of Lie group structures on $\mathbb R^{n+m}$ which is also the common Riemannian mean of infinitely many Riemannian and pseudo-Riemannian metrics.
More precisely, the    mean $\mathfrak{m}=(\mathfrak{m}^1,\dots,\mathfrak{m}^{n+m})$ of  a cloud of $p$ points $\{\sigma_i\}_{i=1,\dots,p}$  where $\sigma_i=(x_i^1,\dots,x_i^{n+m})\in \mathbb R^{n+m},$ is given, for $k=1,\dots,n,$ $q=1,\dots,m,$ by
\beq \label{eq:mean-H-type-final}\mathfrak{m}^k=E(x^k)\;\; \text{ and } \mathfrak{m}^{n+q}=E(x^{n+q})\!\!+\!\!\frac{1}{4}\displaystyle\sum_{j,l=1}^p\gamma_{jl}^q\Big(
  E(x^j)E(x^l)-\frac{1}{p}\displaystyle\sum_{i=1}^px_i^jx_i^l\Big).\eeq
The parameters $(\gamma_{jl}^q)_{i,j=1,\dots,n,\;q=1,\dots,m}$  are real numbers satisfying $\gamma_{jl}^q=-\gamma_{lj}^q.$
\end{theorem}

\begin{proof}
Let $u\in \mathbb R^{n+m}$ and $\{\sigma_i\}_{i=1,\dots,p}$ a cloud of points of $ \mathbb R^{n+m}$ with components  $\sigma_i=(x_i^1,\dots,x_i^{n+m}).$   We endow  $ \mathbb R^n$ with the set of $2$-nilpotent Lie group structures (\ref{eq:H-type}).                     
We denote  by $u^{-1}=(u^1,\dots,u^{n+m}),$ the inverse of   u for the parametric Lie group structures (\ref{eq:H-type}). By the notation $\gamma_{jl}^qu^jx_i^l $, we mean $\sum\limits_{j,l=1}^p\gamma_{jl}^qu^jx_i^l $,so that
\beq
u^{-1}\!\sigma_i\!\!=\!\!\left(\!u^1\!\!+\!x_i^1,\dots,u^n\!\!+\!\!x_i^n, u^{n+1}\!\!+\!\!x_i^{n+1}\!\!+\!\!\frac{1}{2}\gamma_{jl}^1u^jx_i^l,\dots, u^{n+m}\!\!+\!\!x_i^{n+m}\!\!+\!\!\frac{1}{2}\gamma_{jl}^mu^jx_i^l \!\right)
\eeq
and then 
\beqn
 \log(u^{-1}\!\sigma_i)\!\!&=&\!\!\!\!\left(\!u^1\!\!+\!x_i^1,\dots,u^n\!\!+\!\!x_i^n, u^{n+1}\!\!+\!\!x_i^{n+1}\!\!+\!\!\frac{1}{2}\gamma_{jl}^1u^jx_i^l  -{\color{black}\frac{1}{4}}\gamma_{jl}^1(u^j+x_i^j)(u^l+x_i^l), \dots, \right. \cr
  & & \left. u^{n+m}\!\!+\!\!x_i^{n+m}\!\!+\!\!\frac{1}{2}\gamma_{jl}^mu^jx_i^l -{\color{black}\frac{1}{4}}\gamma_{jl}^m(u^j+x_i^j)(u^l+x_i^l)\!\right) \cr
  &=&\!\!\!\!\left(\!u^1\!\!+\!x_i^1,\dots,u^n\!\!+\!\!x_i^n, u^{n+1}\!\!+\!\!x_i^{n+1}\!\!+\!\!{\color{black}\frac{1}{4}}\gamma_{jl}^1(u^jx_i^l
\!\!-\!\!{\color{black}u^lx_i^j }
  -u^ju^l-x_i^jx_i^l), \dots, \right. \cr
  & & u^{n+m}\!\!+\!\!x_i^{n+m}\!\!+\!\!{\color{black}\frac{1}{4}}\gamma_{jl}^m(u^jx_i^l
\!\!-\!\!u^lx_i^j 
  -u^ju^l-x_i^jx_i^l) \Big).
\eeqn
In order  to find the biinvariant mean relative         to  (\ref{eq:H-type}),   we need   to solve the equation 
\beqn\sum\limits_{i=1}^p\log(u^{-1}\!\sigma_i) =0.
\eeqn

From the  equations $pu^k+\sum\limits_{i=1}^px_i^k=0,$ we get
\beq\label{solution-mean-k}
u^k&=&-E(x^k) 
\eeqn
whereas, the equations     $pu^{n+q}\!\!+\!\!\sum\limits_{i=1}^p\Big(x_i^{n+q}\!\!+\!\!\frac{1}{4}\gamma_{jl}^1(u^jx_i^l
\!\!-\!\!u^lx_i^j 
  -u^ju^l-x_i^jx_i^l)\Big)$ lead to
\beqn\label{solution-mean-n+q}
 u^{n+q}\!\!&=&-E(x^{n+q})\!\!+\!\!{\color{black}\frac{1}{4}}\gamma_{jl}^q(-u^jE(x^l)
\!\!+\!\!{\color{black}u^lE(x^j) }
  +u^ju^l+\frac{1}{p}\displaystyle\sum_{i=1}^px_i^jx_i^l)
\nonumber\\ 
&=&-E(x^{n+q})\!\!+\!\!{\color{black}\frac{1}{4}}\gamma_{jl}^q\Big(E(x^j)E(x^l)
{\color{black}-E(x^l)E(x^j) }
  +E(x^j)E(x^l)+\frac{1}{p}\displaystyle\sum_{i=1}^px_i^jx_i^l\Big)
\nonumber\\
&=&-E(x^{n+q})\!\!+\!\!{\color{black}\frac{1}{4}}\gamma_{jl}^q
 \Big( E(x^j)E(x^l)+{\color{black}\frac{1}{p}}\displaystyle\sum_{i=1}^px_i^jx_i^l\Big)\;.
\eeq
The needed  mean  $\mathfrak{m}$ is the inverse relative  to the group product (\ref{eq:H-type}), of   $u^{-1}$ with  components  $(u^k, u^{n+q})$ given         in   (\ref{solution-mean-k})  and  (\ref{solution-mean-n+q}), with  $k=1,\dots,n$ and  $q=1,\dots,m.$
We thus deduce the components $ (\mathfrak{m}^1,\dots,\mathfrak{m}^n, \mathfrak{m}^{n+1},\dots,\mathfrak{m}^{n+m})$ of $m$ as follows
\beq \mathfrak{m}^k=E(x^k)\;\; \text{ and } \mathfrak{m}^{n+q}=E(x^{n+q})\!\!+\!\!\frac{1}{4}\displaystyle\sum_{j,l=1}^p\gamma_{jl}^q\Big(
  E(x^j)E(x^l)-\frac{1}{p}\displaystyle\sum_{i=1}^px_i^jx_i^l\Big).\eeq
\end{proof}

\subsection{More discussions on the new model of parametric means}
In order to apply this mean to a cloud (dataset) of points in the Euclidean space  $\mathbb R^{N}$ for an integer $N\ge    2,$ we have partitioned $N$ into ($n,m$) with  $n+m=N.$    The choice of the partition ($n,m$) of  $N$ depends on the studied problem and is thus left to the user.         
           
Note that, for  each $k=1,\dots,n,$ the   $k$-th component  of   $\mathfrak{m}$,  is  the arithmetic  mean of  the $k$-th components of all the points in the  cloud
\beqn\mathfrak{m}^k =E(x_i^k):=\frac{1}{p}\sum\limits_ix_i^k,\eeqn                                                                                                                                                                                                                                                                                                                                                                                                                                                                                                                                                                               
whereas, for  each $q=1,\cdots,m,$ the $(n+q)$-th component  of   $\mathfrak{m}$,  is the sum of   the arithmetic  mean of  the $(n+q)$-th components of all the points in the  cloud  and a  linear combination of all the 
$E(x^j)E(x^l)-\frac{1}{p}\displaystyle\sum_{i=1}^px_i^jx_i^l$  weighted  by   the parameters $\gamma_{jl}^q.$   The terms $E(x^j)E(x^l)-{\color{black}\frac{1}{p}}\displaystyle\sum_{i=1}^px_i^jx_i^l$  are       related   to the   variance of the dataset,
as will be more explicitly explained in subsequent works.   More importantly, one can adjust, fix or estimate the parameters $\gamma_{jl}^q$ to better fit a problem at hand.  Even better yet, each fixed value of the parameter $\gamma=(\gamma_{jl}^q)$ defines a different $2$-nilpotent (simply connected) Lie group structure on $\mathbb R^{n+m}$ and lives in the smooth manifold of nilpotent Lie algebra structures on $\mathbb R^{n+m}.$ Since they (the $(\gamma_{jl}^q)$'s)   can vary smoothly, one can make them undergo differential or partial differential equations, series and limits, if the studied phenomenon requires so.      
One remarks that when the parameter vanishes, $\gamma_{jl}^q=0$ for any $j,l=1,\dots,n,$ $q=1,\dots,m,$ then  $\mathfrak{m}$ coincides with the arithmetic mean of the dataset ($\sigma_j$).                                                      
                              
Note also  that the  mean $\mathfrak{m}$ in (\ref{eq:mean-H-type-final})  is the common Riemannian mean of all the infinitely many Cartan-Schouten metrics (\ref{eq:metric-form-2}). So, one can choose the best metric to pair with $\mathfrak{m}$, depending on the studied problem. Given that the set of such metrics is also a smooth manifold of dimension   $\frac{1}{2}(m+n)(n+m+1)$ as insured by Theorem \ref{theo:SetOfCartan-SchoutenMetrics}), one has a  wide range of metrics to choose from.                                                                      
This flexibility of choice, fitting,...,  could be a good advantage over other traditional tools such as the ordinary expectation, the least square, etc.


\begin{thebibliography}{99}



\bibitem{amaribook2016} Amari, S.: \emph{Information geometry and its applications.}  Appl. Math. Sci.,194. Springer, (Tokyo), 2016.

\bibitem{amari85} Amari, S.: \emph{Differential-Geometrical Methods in Statistics.} Lecture Notes in Statistics 28 (Springer,
New York, 1985).

\bibitem{aubert-medina} Aubert, A. and Medina, A.: \newblock\emph{Groupes de Lie pseudo-riemanniens plats.}          \newblock Tohoku Math. J.(2)  55, no. 4, 487-506   (2003).                         



\bibitem{benayadi-boucetta-biinvariant-special-connections-Poisson} Benayadi, S. and Boucetta, M.: \newblock \emph{Special bi-invariant linear connections on Lie groups and finite dimensional Poisson structures.}
\newblock Differential Geom. Appl. 36, 66-89 (2014).

\bibitem{Benayadi-Elduque-low-im} 
Benayadi, S. and Elduque, A.:  \newblock {\em Classification of quadratic Lie algebras of low dimension.}  \newblock J. Math.
Phys.  55, 081703 (2014).

\bibitem{benito-2-step-nilp-quadratic} Benito, P.; de-la-Concepci\'on, D.; Rold\'an-L\'opez, J. and  Sesma, I.,
\emph{Quadratic 2-step Lie algebras: computational algorithms and classification.}
J. Symbolic Comput. 94, 70-89 (2019).



\bibitem{Cartan-Schouten} Cartan, E.   and  Schouten, J.A.:  \newblock {\em On the geometry of the group-manifold of simple and semi-simple groups.}  \newblock Proc. Akad.
Wekensch, Amsterdam 29 (1926) 803-815.




\bibitem{diatta-manga-sy-NLAGA} Diatta, A.; Manga B. and Sy, F.:  \newblock {\em On dual quaternions, dual split quaternions and Cartan-Schouten metrics on perfect Lie groups.}   
\newblock Trends Math.
Birkh\"auser/Springer, Cham, 2024, pp 317-339.


 \bibitem{Diatta2000}
                Diatta, A., \newblock {\em G\'eom\'etrie de Poisson et de  contact des espaces homog\`enes.} \newblock  Ph.D. Thesis. University Montpellier 2, France (2000).



\bibitem{diatta-contact}  Diatta A.,  \newblock {\em Left Invariant Contact Structures on Lie Groups.} \newblock Diff. Geom. Appl. {\bf 26}  (2008),  no. 5, 544-552. 




\bibitem{eberleinbook}
Eberlein, P.   \newblock  {\em 
Geometry of 2-step nilpotent Lie groups.}  \newblock  Modern dynamical systems and applications, 67-101.
Cambridge University Press, Cambridge, 2004

\bibitem{eberleinENS}
Eberlein, P.  \newblock  {\em Geometry of 2-step nilpotent groups with a left invariant metric.}
 \newblock  Ann. Sci. Ecole Norm. Sup. (4) 27, no. 5, 611-660 (1994).

\bibitem{eberleinTransAMS}
Eberlein, P.  Geometry of 2-step nilpotent groups with a left invariant metric. II
Trans. Amer. Math. Soc. 343, no. 2, 805-828  (1994).

\bibitem{garcia}
Garc\'ia T. N.
 \newblock  {\em Gromov-Hausdorff limit of Wasserstein spaces on point clouds.}
Calc. Var. Partial Differential Equations 59, no.2, Paper No. 73, 43 pp (2020).
\bibitem{godoy}
Godoy M.M.; Kruglikov, B.; Markina, I. and Vasil'ev, A.:  \newblock 
{\em Rigidity of 2-step Carnot groups.}  \newblock  J. Geom. Anal.28, no.2, 1477-1501 (2018).
\bibitem{kumariKNN}
 Kumari, S. and Pestov, V. G.  \newblock {\em 
Universal consistency of the k-NN rule in metric spaces and Nagata dimension. II}   \newblock  ESAIM Probab. Stat. 28, 132-160 (2024).


\bibitem{figueroa-ofarrill-elf-dual} Figueroa-O'Farrill, J. and Stanciu, S.: \newblock {\em On the structure of symmetric self-dual Lie algebras.} \newblock J. Math. Phys. 37 (8), 4121-4134  (1996).




\bibitem{gallier-quaintance} Gallier J. and Quaintance, J.: \newblock{\em Differential geometry and Lie groups - a computational perpective.} \newblock Geometry and Computing, 12. Springer, Cham, 2020, 777pp.


\bibitem{bi-invariant-and-noninvariant-metrics-ghanam-hindeleh-thompson}
Ghanam, R.; Hindeleh, F. and Thompson, G.; \emph{Bi-invariant and noninvariant metrics on Lie groups.} J. Math. Phys. 48 (2007), no. 10, 102903



               

\bibitem{lauritzen87}Lauritzen,  S. L.: \emph{Statistical manifolds} In Differential Geometry in Statistical Inferences. IMS Lecture Notes Monogr. Ser., 10, Inst. Math. Statist.,
Hayward California, 1987, pp. 96-163.

\bibitem{van-le2005}     L\^e, H. V.:\newblock\emph{Statistical manifolds are statistical models.} \newblock J. Geom. 84, no. 1-2, 83-93 (2005).  

\bibitem{Lorentzi-Pennec}
Lorenzi, M. and  Pennec, X.: \emph{Geodesics, parallel transport and one-parameter subgroups for diffeomorphic image registration.} 
Int. J. Comput. Vis. 105, no. 2, 111-127 (2013). 

\bibitem{matumoto1993}   Matumoto, T.: \newblock\emph{Any statistical manifold has a contrast function. - On the $C^3$-functions taking the minimum at the diagonal of the product manifold.} \newblock Hiroshima Math. J. 23, 327-332   (1993).          
\bibitem{Matsuzoe2007} Matsuzoe,  H.: \emph{Geometry of statistical manifolds and its generalization.} In 
Proceedings of the 8th International Workshop on Complex Structures and Vector Fields. World Scientific, 2007, pp. 244-251

\bibitem{Matsuzoe2010} Matsuzoe,  H.: \emph{Statistical manifolds and affine differential geometry.} Advanced Studies in Pure Mathematics 57, 2010
Probabilistic Approach to Geometry
pp. 303-321.
 


\bibitem{medina85} Medina, A. \newblock{\em Groupes de Lie munis de m\'etriques biinvariantes.}  \newblock T\^ohoku Math. J. 37, 405-421 (1985).

\bibitem{medina-revoy-ENS} Medina, A. and Revoy, Ph.:   \newblock {\em Alg\`ebres de Lie et produit scalaire invariant.} \newblock Ann. Scient. Ec. Norm. Sup., 4e serie, {\bf 18}, no 3,  553-561 (1985)


\bibitem{Milnor76} Milnor, J.: \emph{ Curvatures of left invariant metrics on Lie groups.}  Adv. Math. 21, no. 3, 293-329 (1976).

 \bibitem{miolane-pennec2015} Miolane, N. and Pennec, X.; \emph{Computing bi-invariant pseudo-metrics on Lie groups for consistent statistics.} Entropy 17, no. 4, 1850-1881 (2015).

\bibitem{pennec1} Miolane, N. and Pennec X.; \emph{Statistics on Lie groups: a need to go beyond the pseudo-Riemannian framework.}
  AIP Conference Proceedings 1641 (1) 59-66 (2015).



 \bibitem{Nomizu79} Nomizu, K.: \emph{Left-invariant Lorentz metrics on Lie groups.} Osaka J. Math. 16, no. 1, 143-150 (1979).

\bibitem{Revoy-3forms} Noui, L. and Revoy, Ph., \emph{ Alg\`ebres de Lie orthogonales et formes trilin\'eaires altern\'ees.} Comm. Algebra 25, no. 2, 617-622  (1997).



\bibitem{pennec-biinvariant-means} Pennec, X.: \emph{Bi-invariant means on Lie groups with Cartan-Schouten connections.} Geometric
Science of Information, 59-67, Lecture Notes in Comput. Sci, 8085, Springer, Heidelberg, 2013.

\bibitem{pennec2} Pennec, X. \emph{Intrinsic statistics on Riemannian Manifolds: Basic tools for geometric
measurements.} J. Math. Imaging Vision 25 (1), 127-154 (2006).


\bibitem{phillips} Phillips, N.C.: \emph{How many exponentials?} 
Amer. J. Math. 116, no. 6, 1513-1543 (1994). 

\bibitem{rawashdeh-thompson} 
Rawashdeh, M. and Thompson, G. \newblock{\em The inverse problem for six-dimensional codimension two nilradical Lie algebras.}  \newblock J. Math. Phys. 47, no. 11, 112901, 29 pp  (2006). 



\bibitem{peyghan2022} Samereh, L.; Peyghan, E.  and  Mihai, I.: \emph{On almost Norden statistical manifolds.} Entropy 24, no. 6, Paper No. 758, 10 pp  (2022).



\bibitem{shima1980} Shima, H. \newblock\emph{Homogeneous Hessian manifolds.}    Ann. Inst. Fourier (Grenoble) 30 (1980), 91-128.                        

\bibitem{strugar-thompson} Strugar, I.; Thompson, G.: \newblock{\em 
Inverse problem for the canonical Lie group connection.}  \newblock 
Houston J. Math. 35, no. 2, 373-409 (2009).

\bibitem{FatimataSy} Sy, F.;   \emph{Restricted Inverse problem of  Langrangian dynamics for the Cartan-Schouten canonical connection and applications.} 
 Ph.D Thesis. Universit\'e C.A. Diop. In preparation.

\bibitem{thompson3D}
Thompson, G. \newblock{\em Metrics compatible with a symmetric connection in
dimension three.} \newblock J. Geom. Phys. 19, 1-17 (1996).

\bibitem{tondeur} Tondeur, Ph.: 
\newblock{\em Sur certaines connexions naturelles d'un groupe de Lie. Applications.} \newblock  S\'eminaire Ehresmann. Topologie et g\'eom\'etrie diff\'erentielle, tome 6 (1964), exp. no 5, p1-9.


\bibitem{watanabe2009}Watanabe, S.: \newblock\emph{Algebraic geometry and statistical learning theory.} Cambridge Monogr. Appl. Comput. Math., 25. Cambridge University Press, Cambridge, 2009.             
\bibitem{Zefran} Zefran,M.;  Kumar, V. and Croke, C.;  \emph{Metrics and Connections for Rigid-Body Kinematics.}
The International Journal of Robotics Research, 18 (2) 242 (1999).

\end{thebibliography}
\end{document}